\documentclass[a4paper,reqno]{amsart}

\pdfoutput=1

\usepackage{amsmath,amsthm}

\usepackage{amssymb,latexsym}

\usepackage[mathscr]{euscript}

\usepackage{enumerate}

  \usepackage{graphics,tikz}
\usetikzlibrary{decorations.pathreplacing}
 
\usepackage{xcolor}

\usepackage{caption}
\usepackage{hyperref}

\newcommand{\cal}{\mathcal}

\newcommand{\FF}{{\cal F}}

\newcommand{\BR}{{\mathbb R}}

\newcommand{\si}{\sigma}
\newcommand{\om}{\omega}

\newcommand{\E}{{\mathbb E}}
\newcommand{\bbE}{{\mathbb E}}
\newcommand{\R}{{\mathbb R}}
\newcommand{\bbR}{{\mathbb R}}
\newcommand{\al}{\alpha}
\newcommand{\la}{\lambda}
\newcommand{\bn}{\bf n}
\newcommand{\eps}{\epsilon}
\newcommand{\ga}{\gamma}

\newcommand{\essinf}{\mathop{\mathrm{ess\,inf}}}

\newtheorem{theorem}{\bf Theorem}[section]
\newtheorem{proposition}[theorem]{\bf Proposition}
\newtheorem{lemma}[theorem]{\bf Lemma}
\newtheorem{corollary}[theorem]{\bf Corollary}

\theoremstyle{definition}
\newtheorem{definition}[theorem]{Definition}

\newtheorem{remark}[theorem]{Remark}

\numberwithin{equation}{section}

\begin{document}

\title[Maximum principles and    Hopf  lemma for integro-differential operators]{On the maximum principles and the quantitative version of the Hopf  lemma for
   uniformly elliptic integro-differential operators}

\author {Tomasz Klimsiak  
\and  Tomasz Komorowski}

\address[Tomasz Klimsiak]{Institute of Mathematics, Polish Academy Of Sciences,
ul. \'{S}niadeckich 8,   00-656 Warsaw, Poland, \and Faculty of
Mathematics and Computer Science, Nicolaus Copernicus University,
Chopina 12/18, 87-100 Torun, Poland, e-mail :{\tt tomas@mat.umk.pl}}

\address[Tomasz Komorowski]{Institute of Mathematics,
  Polish Academy Of Sciences, ul. \'Sniadeckich 8, 00-636 Warsaw, Poland
e-mail:{ \tt tkomorowski@impan.pl}}

\date{}
\maketitle
\begin{abstract}
In the present paper we prove   estimates on
{subsolutions of the equation $-Av+c(x)v=0$}, $x\in D$, where $D\subset \bbR^d$ is a
domain (i.e. an open and connected set) and $A$ is an
integro-differential operator of the Waldenfels type,  whose
differential part  satisfies the uniform
ellipticity condition on compact sets.  In general,  the coefficients of the
operator need not be continuous but only bounded and Borel measurable. Some of our results may be
considered      ''quantitative'' versions of the Hopf lemma, as they provide
the lower bound 
on  the outward normal directional derivative at the maximum point of
a subsolution 
in terms of its value at the point. We
shall also show lower bounds on the subsolution around its maximum
point by the principal eigenfunction associated with $A$ and the domain.  Additional results, among them the weak and
strong maximum principles, the weak Harnack inequality are also proven.
\end{abstract}




\section{Introduction}
\label{sec1}

The maximum principle and the Hopf lemma express some of the most fundamental properties of
solutions of elliptic partial differential equations. 
Consider a     differential  operator
\begin{equation}
\label{L}
Lu(x):=\frac12\sum_{i,j=1}^dq_{i,j}(x)\partial_{x_i,x_j}^2u(x)+\sum_{i=1}^db_{i}(x)\partial_{x_i}u(x),\quad
u\in C^2(\bbR^d)
\end{equation}
that is {\em uniformly elliptic} on compact
  sets, i.e. the coefficients
 $q_{i,j} $, $b_j\ $,  $i,j=1,\dots,d$ are bounded, Borel measurable
  and 
  the matrix 
${\bf
  Q}(x)=[q_{i,j}(x)]_{i,j=1}^d$  is {\em uniformly positive
definite} on compact sets, see \eqref{la-K} below.
 
{The strong
maximum principle, see e.g. \cite[Theorem 9.6, p. 225]{gilbarg}, can
be formulated as follows. 
Suppose that $D$ is a   domain (an open and connected set) and $c$ is
a non-negative, bounded and measurable function on $
D$. Assume
furthermore that  a
continuous function $u$, belonging to $W^{2,d}_{\rm
  loc}(D)$ (i.e. its two generalized derivatives 
are $L^d$-integrable on compacts) and satisfying
\begin{equation}
\label{010610-20}
-Lu(x)+c(x)u(x)\le0,\quad x\in D
\end{equation}
 admits a   maximum $M\ge0$ in $D$. Then, $u$ has to be constant,
 i.e. $u(x)= M$, $x\in D$. A function satisfying \eqref{010610-20} is
 called a {\em subsolution} of the equation $(-L+c)f=0$ in $D$.}

If $\partial D$ - the boundary of the domain - is
sufficiently regular, say $C^2$ smooth and $u$ is non-constant, then
the {\em Hopf lemma}, see
e.g. \cite[Theorem 7, p. 65]{weinberger}, asserts that $\partial_{\bf
  n}u(\hat x)$ - the  directional
derivative of $u$ in the outward normal to $\partial D$  - at the
maximal point $\hat x$ has to be strictly positive.

It has been shown in \cite{boko20} that if  $\partial D$ is $C^2$
smooth and $\underline c:=\inf_{x\in D} c(x)>0$, then  there exists a
constant $a>0$, depending only on the ellipticity constant on $\bar D$, $\underline c$, the $L^\infty$ bound on the
coefficients of $L$ and $c$, and $D$ such that
\begin{equation}
\label{quant-hopf}
\partial_{\bf
  n}u(\hat x)>a u(\hat x)
\end{equation}
for any $u$ satisfying  \eqref{010610-20} that is not identically equal to $0$.
The above estimate  could be  viewed as a form of a {\em
  quantitative Hopf lemma}.
An elementary example, using harmonic functions, shows that \eqref{quant-hopf} is
false if $ c\equiv 0$. 

One of the objectives of the present paper is to investigate an
analogue of the estimate \eqref{quant-hopf} and its  generalizations  for non-local operators.
More precisely, we consider non-divergence form integro-differential
operators  
\begin{equation}
\label{011703-20}
Au(x)=Lu(x)+Su(x),\quad u\in C^{2}(\BR^d)\cap C_b(\BR^d),\,x\in\BR^d,
\end{equation}
where $L$ is given by \eqref{L}
and 
\begin{equation}
\label{S}
Su(x):=\int_{\BR^d}\left(u(x+y)-u(x)-\sum_{i=1}^d\frac{y_i\partial_{x_i}u(x)}{1+|y|^2}\right)N(x,dy),\quad x\in\BR^d.
\end{equation}
The coefficients of the operator $L$ are as in the foregoing and the
kernel $N$ (sometimes referred to as the L\'evy kernel) assigns to each $x\in\bbR^d$   a
$\si$-finite Borel measure $N(x,\cdot)$, which  satisfies  
\begin{equation}
\label{Nx1}
N_*:=\sup_{x\in \bbR^d}\int_{\bbR^d} \min\{1,|y|^2\} N(x,dy)<+\infty.
\end{equation}
Such operators are sometimes called of the  {\em Waldenfels type},
see e.g. \cite[Chapter 10]{Taira2}. A sufficiently regular function
$u:D\to\bbR$, see Section \ref{sec2.3a} for a rigorous definition, is
called a {\em subsolution (resp. supersolution)}   of  equation
\begin{equation} 
\label{eq3.1a}
(-A+c)f=0.
\end{equation}
 if $-Au(x)+c(x)u(x)\le 0$ (resp. $-Au(x)+c(x)u(x)\ge 0$) in
 $D$. Throughout the paper we shall always assume that $c:D\to\bbR$
 is nonegative,  bounded and Borel measurable.

Let
\begin{equation}
\label{cSD}
\mathcal S(D):=\bigcup_{x\in D}\mathcal S_x, \quad \mbox{where $\mathcal S_x:=\big[x+y: y\in \mbox{supp}\, N(x,\cdot)\big].$}
\end{equation} 
The set $\mathcal S_x$ describes the "range of non-locality" of $A$ at
$x$, i.e. the evaluation of $Au(x)$
depends on the values of $u$ in an arbitrarily small neighborhood of $x$, and
in $\mathcal S_x$. 

Then, the following version of the ''quantitative Hopf lemma'' holds,
see Theorem \ref{thm:main1a} below for a  precise
statement.
\begin{theorem} 
\label{thm:main1a-informal}
Suppose that   $\underline c>0$, 
 the domain $D$ is bounded and satisfies the  uniform interior
ball condition (see Definition \ref{def2.3}), $u$ - a subsolution of
\eqref{eq3.1a} - attains its  maximum  over $\mathcal
S(D)\cup \overline{D}$
at some $\hat x\in \partial D$  and is not identically equal to $0$. Then, there exists $a>0$ that depends
  only on $D$, ellipticity constant,
     $\underline c$, the $L^\infty$ norm of the coefficients and $N_*$
such that estimate
\eqref{quant-hopf} holds
\end{theorem}
We can relax the assumption that $\underline c>0$. Namely, see Theorem
\ref{thm:main1} below, we   admit that   
$\underline c=0$ but $c(x)$
does not vanish a.e. in $D$. However, in that case we have to allow $a>0$ in
\eqref{quant-hopf}  to depend
also on the
operator $A$ (in particular  on the lower bound of its Green
function), not just on some characteristics of its coefficients as in
Theorem \ref{thm:main1a-informal}.

In our further generalizations of estimate \eqref{quant-hopf}, see Theorems
\ref{pee1} and \ref{pee} below, we compare the increment  $u(\hat
x)-u(x)$ with the value of the principal eigenfunction $\varphi_D(x)$ associated
with the operator $A$ on the domain $D$, i.e. the unique  normalized eigenfunction
that is strictly positive in $D$, see Theorem
\ref{thm020809-20}. It turns out (see Theorem \ref{pee}) that, with no 
assumption on the regularity of $D$,  the following bound holds
\begin{equation}
\label{012804-20zz}
u(\hat x)-u(x)\ge \frac{\varphi_D(x)}{2e
  \|\varphi_D\|_\infty }\left[\frac{\underline
  cu(\hat
x)}{\la_D+\underline c}+\frac{{\rm essinf}_D(Au-cu)}{\la_D+\|c\|_\infty}\right],\quad x\in D.
\end{equation} 
Here $\la_D$ is the principal eigenvalue corresponding to $\varphi_D$
and $\|\cdot\|_\infty$ is the usual supremum norm.

In fact, instead of the essential infimum of $Au-cu$ in \eqref{012804-20zz}, we can use its
average over $D$ with respect to some measure equivalent with   the
Lebesgue measure. Namely, see Theorem
\ref{pee1}, there exist $\psi(x)$   continuous,
strictly positive in $D$ and $\chi(x)$  strictly positive, Borel
measurable  such that any subsolution $u,$ described in the
foregoing, satisfies
\begin{equation}
\label{012804-20-1zz}
u(\hat x)-u(x)\ge \frac{\underline
  c\varphi_D(x) u(\hat
x)}{2e
  \|\varphi_D\|_\infty (\la_D+\underline c)}+\psi(x)\int_D(Au-cu)\chi\,dx,\quad x\in D.
\end{equation}
We have already mentioned that estimates
\eqref{012804-20zz} and \eqref{012804-20-1zz} are obtained without
making assumptions on the regularity of the domain $D$. However, if we suppose that
  both the uniform interior and exterior ball conditions hold, see
Definitions \ref{def2.3} and \ref{def3.1z}, then we can choose
$\psi(x)=\delta_D(x)$, where $\delta_D(x) :={\rm dist}(x,D^c)$, see
Theorem \ref{cor.d21.20.1}. Furthermore, if  some  
regularity of the coefficients of $A$ is allowed, then,  in addition
to the above choice of $\psi(x)$,  we
can admit also $\chi(x)=a\delta_D(x)$, with some $a>0$, see Theorems
\ref{cor.d21.20.1a} and \ref{thm013009-20}.

It is also worthwhile to mention that the estimates  
\eqref{012804-20zz}
and \eqref{012804-20-1zz}  hold even  without  the uniform
ellipticity hypothesis. They  are related to ergodic properties  of the
Markov process,   associated with the operator $A$ via the respective
martingale problem, see Section \ref{sec4} below, e.g. its uniform
conditional ergodicity, or
 intrinsic ultracontractivity of its transition semigroup. This shall be the topic of  our forthcoming paper \cite{KK20a}.

In addition to the results mentioned above, in
the present paper  we  also obtain a
weak maximum principle for operators of the form \eqref{011703-20},
see Proposition \ref{prop012804-20}. Namely, if $u$ is a subsolution
of \eqref{eq3.1a} in $D$, then 
\begin{equation}
\label{032804-20zz}
\sup_{x\in D}u(x)\le  \sup_{y\in\mathcal (S(D)\setminus
  D)\cup \partial D}u^+(y).
\end{equation}
This result holds without assuming that $D$ is connected.
Using the weak form of the maximum principle we establish its strong
version and the Hopf lemma, see Theorem \ref{strong-max}. In addition
we prove also the Bony type maximum principle and {a form} of the weak Harnack
inequality, see    Theorems
\ref{thm:main3a} and \ref{thm:main4a}, {respectively}. 
{The
  latter asserts the existence of  a strictly positive, bounded, Borel measurable
  function $\chi$ on $D$ such that
 for any open $V$, with $\bar V$  compact and  $\bar V \subset
 D$,  there exists a constant $C>0$   such that
for any  $u$ - a  non-negative supersolution to \eqref{eq3.1a} -  we have
\begin{equation}
\label{830104-20aabba}
\inf_{x\in V} u(x)\ge C \int_V\chi u\,dx.
\end{equation} 
In the case when $D$ is regular, e.g. satisfies both the interior 
and exterior ball conditions (see Definitions \ref{def2.3} and
\ref{def3.1z}) and  the formal adjoint to $A$ satisfies appropriate
hypotheses, see Section \ref{sec5.2z}, one can take $\chi(x)=\delta_D(x)$, $x\in
D$, see Remarks \ref{rmk9.3} and \ref{rmk9.4} below, and conclude the
''usual'' weak Harnack inequality, see e.g. \cite[Theorem 8.18,
p. 194, or Theorem 9.22, p.246]{gilbarg},
\begin{equation}
\label{830104-20aabbb}
\inf_{x\in V} u(x)\ge C\int_Vu\,dx.
\end{equation}}

{Recall that  the  (strong) Harnack inequality, asserts the existence of $C>0$, for which   
$\inf_{x\in V} u(x)\ge C \sup_{x\in V} u(x)$ for  all non-negative supersolutions. This type of inequality
need not hold in general for operators of the type
\eqref{011703-20} and for its validity some 
  additional restrictions on the kernel $N$ are needed, see \cite[Section
7]{Foondun} for more details}.

Let us briefly discuss the relation between our work and some  
results  existing in the
literature. The maximum principles both weak and strong, together with
the Hopf lemma can be found   e.g. in  \cite[Section I.4]{BCP}, \cite[Section 10.2]{Taira2} and
\cite[Appendix C]{Taira6}. However,  the versions of the maximum
principles we are aware of  do not cover the case
considered in our paper.  For example, the formulations of
\cite{Taira2,Taira6} concern only the case  when coefficients are
$C^\infty$ smooth, while \cite{BCP} requires some regularity
assumptions on the L\'evy kernel, see Section I.2.3 and condition
(I.4.1) ibid.    {The closest formulation
to part 1) of our Theorem \ref{strong-max},   we have encountered
in the
recent  paper of Taira \cite{Taira5}. In ibid.  the author has additionally
assumed that  ${\cal S}_x(D)=\overline D$ for $x\in D$ (no jumps
outside $\overline D$) and that there exists a reference measure, with respect to which, all
the L\'evy kernels are absolutely continuous.} The  
 regularity of the resulting densities is also assumed.
The domain $D$ is to be of class $C^{1,1}$ and $u\in W^{2,p}(D)$ for some
 $p>d$. 

Our formulation of  the Hopf lemma, see part 2) of Theorem
\ref{strong-max} and in particular Corollary \ref{thm:harnack},
resembles somewhat  
 Proposition 1.34 of \cite{hl}, which actually uses the phrase
 ''quantitative Hopf lemma''.
The result of ibid.  holds     for harmonic
functions on a ball $B$    and
asserts that there exists $a>0$ such
that for any    bounded harmonic function $u$ in $B$, we have 
\begin{equation}
\label{eqi.3}
\partial_{\mathbf n}u(\hat x)\ge a(u(\hat x)-u(x_0))
\end{equation}
with $\hat x\in \partial B$ a maximum point of $u$ in $\bar B $ and
$x_0$ - the center of the ball.
The counterpart of this
estimate for   $A$-harmonic functions, provided that the operator has
the Harnack property (see Definition \ref{hnk}), is stated in
  \eqref{010104-20b} below.



{Estimate  \eqref{012804-20-1zz} of Theorem \ref{pee1}} can be compared with the Morel and Oswald    version of the 
Hopf lemma proved by Brezis and
Cabr\'e in  \cite{BC}, see Lemma 3.2.  It asserts that for a smooth and
bounded domain $D$ there exists
$a>0$   such that for any   {superharmonic} function  $u$,
vanishing on the boundary, we have
\begin{equation}
\label{eqi.5}
 {u(x)\ge a\delta_D(x) \int_D(-\Delta u)(y) \delta_D(y)\,dy},\quad x\in D.
\end{equation}
The estimate is sometimes referred to as  the {\em uniform Hopf
lemma} (see e.g. \cite{BDR,dmo,op}).  

We prove an analogue of
\eqref{eqi.5}, see Theorem \ref{thm013009-20}, for $C^2$-regular
 domains,   $u(x)$ - a subsolution of \eqref{eq3.1a} - and
operator $A$ of the form \eqref{011703-20}, whose coefficients  {of}
the second derivative, as
well as the respective coefficients of its formal adjoint belong to  the VMO class,
see Definition \ref{df5.6}.
From \eqref{eqi.5} one can easily conclude the {\em weak Harnack
  inequality} (see e.g. \cite[Theorem 9.22]{gilbarg}) that is shown in
the present paper in Theorem \ref{thm:main4a}.

The principal tools used in our analysis are: the martingale
problem associated with operator $A$, see Section \ref{SMM}, and the Feynman-Kac
formula for the respective canonical process $(X_t)_{t\ge0}$, see Proposition \ref{prop010204-20}. They allow us to obtain
a probabilistic representation for a subsolution of the equation
\eqref{eq3.1a}, see Definition \ref{df3.1}, in terms of the expectation
of an appropriate functional of the process, see
\eqref{010204-20dfg}. We exploit various properties the process
$(X_t)_{t\ge0}$, such as its irreducibility and the strong Feller property, to derive our results.

{The body   of the paper is essentially divided into two
  parts.  The first one, consisting of Sections \ref{sec2} to 
\ref{sec5z}, besides presenting some introductory material,  is
primarily devoted to the exposition of the results contained in
the article. Their proofs have been placed in the second part,
containing Sections \ref{sec6} through   \ref{sec5}.   In more detail, in Section 2 we introduce
the basic notation and definitions.     The   formulations of the
main results are contained in Sections \ref{sec3} and
\ref{sec5z}. In Section \ref{sec4}  we introduce some basic probabilistic
tools used  throughout the  paper, such as: the already
mentioned  
martingale problem associated with $A$, the (probabilistic) semigroup
and resolvent corresponding to the canonical process, the gauge
functional and the principal eigenvalue and eigenfunction associated
with $A$ and domain $D$. They are used in formulations of the results
of Section \ref{sec5z}. In Section \ref{sec6} we present the proof of
the Feynman-Kac formula, which is our basic tool in showing of
virtually  all the results of  the paper. We have also placed in that
section the proof of the weak maximum principle (Proposition
\ref{prop012804-20}) as it almost directly follows from the
aforementioned formula.  Section \ref{sec5.1} discusses various
properties of the resolvent that will be useful in our further
arguments. Section \ref{sec8} contains   auxiliary facts and
notions used in the proofs of our main results. Finally,  
the main results are proved in Section \ref{sec5}}.

\section{Preliminaries}
\label{sec2}

\subsection{Basic notation}

Given   a metric space $E$  we denote by ${\cal B}(E)$ its Borel $\si$-algebra. Let
$B_b(E)$ ($B^+_b(E)$) be the space of all (non-negative) bounded Borel
measurable functions and let $C_b(E)$ ($C_c(E)$) be the space of all
bounded  continuous (compactly supported) functions on $ E$.
Furthermore by ${\cal M}(E)$ we denote the set of all Borel measures
on $E$. 
Suppose that $\mu,\nu\in {\cal M}(E)$. We say that $\mu$ dominates
$\nu$ and write $\nu\ll\mu$ if all null sets
for $\mu$ are also null for $\nu$.
We say that measures are equivalent and write $\mu\sim\nu$ iff
$\mu\ll\nu$ and $\nu\ll\mu$.

Given a point $x\in E$ and $r>0$ we 
let $B(x,r)$ be the open ball of radius $r$ centered at $x$ and 
$\bar B(x,r)$ be its closure.
As it is customary
for a given function $f:E\to\R$ we denote  $\|f\|_\infty=\sup_{x\in E}|f(x)|$. 
For a subset $B\subset E$ we let $B^c:= E\setminus B$ and $\bar B$
be its closure.

Suppose that $B$ is an arbitrary set. For functions $f,g:B\to[0,+\infty)$
we write $f\preceq g$ on $B$ if there exists number $C>0$, i.e., \textit{constant}, such that
$$
f(x)\le Cg(x),\quad x\in B. 
$$
Furthermore, we write $f\approx g$ if both $f\preceq g$ and $g\preceq f$.

Assume now that $D\subset \bbR^d$ and $V\subset D$ are open. We shall write
$V\Subset D$ if $\bar V$ is compact and $\bar V\subset D$.
Furthermore we 
let $C^m(D)$, $m\ge1$ be the class of $m$-times
continuously differentiable functions in $D$. By $C_0(D)$ we denote
the subset of  $C(D)$ that consists of functions extending
continuously to $\bar D$ by letting $f(x)\equiv0$, $x\in\partial D$ -
the boundary of $D$.

Denote by $m_d$ the $d$-dimensional Lebesgue measure on
$\bbR^d$. Sometimes we omit writing the subscript, when the dimension
is obvious from the context. By $dx$ we denote the volume element
corresponding to the Lebesgue measure. {A function $f\in B^+(D)$ is
called non-trivial if $\int_Df dx>0$.}

For $p\in[1,+\infty)$ we denote by  $L^p(D)$  ($L^p_{\rm loc}(D)$) the space of functions
that are integrable with their $p$-th power on 
$D$ (any compact subset of $D$). By  $W^{k,p}(D)$
($W^{k,p}_{\rm loc}(D)$) we denote the Sobolev space of
functions whose $k$ generalized derivatives belong to $L^p(D)$
($L^p_{\rm loc}(D)$).  Given $f\in L^p(D)$ and $B\in {\cal B}(D)$ we
denote by $\essinf_{B}f$ the usual essential infimum of $f$ in
$B$. We define  the {\em essential limit inferior} of $f$
at  $x_0\in D$ as
\begin{equation}
\label{essinf}
\liminf_{x\rightarrow x_0} {\rm ess}
\,f(x):=\lim_{r\to0+}\essinf_{B(x_0,r)\subset D}f.
\end{equation}


\subsection{Second-order, elliptic integro-differential operators}

\label{sec2.1}

Consider
 an elliptic  integro-differential operator of
 the form
\begin{equation}
\label{011703-20i}
Au(x)=Lu(x)+Su(x),\quad u\in C^{2}(\BR^d)\cap C_b(\BR^d),\,x\in\BR^d,
\end{equation}
where $L$, $S$ are   a differential and
integro-differential operators, given by \eqref{L} and \eqref{S} respectively.
Throughout the paper we shall always assume that the  coefficients satisfy the following hypotheses:
\begin{itemize}
\item[A1)] $q_{i,j}\in B_b(\BR^d)$, $b_j\in B_b(\BR^d)$,  $i,j=1,\dots,d$ and 
  we suppose that the differential operator $L$ is {\em uniformly elliptic} on compact
  sets, i.e. 
  the matrix 
${\bf
  Q}(x)=[q_{i,j}(x)]_{i,j=1}^d$  is {\em uniformly positive
definite} on compact sets. The latter means that for any compact set $K\subset \R^d$
there exists $\ga_K>0$ such that 
 \begin{equation}
\label{la-K}
\ga_K|\xi|^2\le \sum_{i,j=1}^dq_{i,j}(x)\xi_i\xi_j,\quad x\in K,\, \xi=(\xi_1,\ldots,\xi_d)\in\BR^d.
\end{equation}
\item[A2)]  $N:\bbR^d\to {\cal M}(\bbR^d)$, called the {\em
  L\'evy
kernel},  is a $\si$-finite Borel measure valued function such that,
{the mapping $x\mapsto N(x,A)$ is Borel measurable for each $A\in {\cal
  B}(\bbR^d)$ and}
\eqref{Nx1} holds.
\end{itemize}
We let
\begin{equation}
\label{MA}
M_A:=\sum_{i,j=1}^d\|q_{i,j}\|_\infty+\sum_{i=1}^d\|b_{i}\|_\infty+N_* 
\end{equation}


{Suppose that $D$ is an open set.
Obviously in order to define $Au(x)$ for $x\in D$ it suffices to
assume that $u\in C^2(D)\cap C_b(\bbR^d)$. In fact, 
$Au$ can be defined as an element of $L^p(D)$, using \eqref{011703-20}--\eqref{S}, even if  $u\in W^{2,p}_{\rm loc}(D)\cap
C_b(\bbR^d)$ when $p>d$. 
This is possible due to a well known estimate, see \cite[Lemme 1,
p. 361]{bony1}:  for any $p>d$ there
exists $C>0$ such that
$$
\|U[v]\|_{L^p(\bbR^d)}\le C\|\nabla^2 v\|_{L^p(\bbR^d)},\quad v\in W^{2,p}(\bbR^d).
$$
Here
$$
U[v](x):=\sup_{|y|>0}|y|^{-2}\left|v(x+y)-v(x)-\sum_{i=1}^dy_i\partial_{x_i}v(x)\right|.
$$}

\subsection{A subsolution, supersolution and solution of an elliptic type
integro-differential equation}

\label{sec2.3a}

\begin{definition}
\label{df3.1}
 Suppose that $c,{g}:\BR^d\to\BR$.  A function $u: \BR^d\to\BR$ is called a {\em subsolution}
of the equation
\begin{equation} 
\label{eq3.1}
(-A+c)f={g}
\end{equation}
on an open set $D\subset \bbR^d$
if $u\in W^{2,p}_{\rm loc}(D)\cap C_b(\BR^d)$ for some $p>d$ and
\begin{equation}
\label{subsol}
-Au(x)+c(x)u(x)\le g(x),\quad \mbox{for a.e. } x\in D.
\end{equation}
{We say that $u$ is a {\em supersolution} of equation \eqref{eq3.1a}
if $-u$ is  {a}
  subsolution. Furthermore, $u$ is a  {\em solution}
if it is both sub- and supersolution.}
\end{definition}
{Throughout the paper, we shall mostly deal with    homogeneous  equations, i.e. 
the case when $g=0$.}

\bigskip
Given a bounded and measurable function $c:D\to\bbR$ we let
\begin{equation}
\label{c-c}
\bar c=\sup_{x\in D} c(x),\quad \underline{c}:=\inf_{x\in D}
c(x)\quad\mbox{and}\quad \langle c\rangle_D:=\int_Dc\,dx.
\end{equation}
Concerning the coefficient $c(\cdot)$ it is assumed to satisfy $c\in
B_b(D)$ and  either one of the following hypotheses:
\begin{itemize}
\item[A3)] 
it is non-negative on $D$, or 
\item[A3')] 
it is both non-negative on $D$ and 
$
\langle c\rangle_D>0,
$ (i.e. $c\not\equiv 0$ a.e.), or 
\end{itemize} 
\begin{itemize}
\item[A3'')] 
$
\underline{c}>0.
$
\end{itemize}

Below we formulate some hypotheses concerning an open
set 
$D$ that shall be used in various results presented  {in} the
paper.  Throughout the paper, unless stated otherwise (notably
in 
 Proposition \ref{prop012804-20})   it is
 supposed to be  a   {\em domain}, i.e.  a bounded, open and 
  connected subset of $\R^d$. 

\begin{definition}[Interior ball condition]
\label{def2.3}
For a given domain $D$ 
and $\hat x\in\partial D$, we say that the {\em interior
ball condition} is satisfied at $\hat x$, if there exists $B$ -  an open
ball of a positive
radius - contained in $D$ and such that $\hat x\in \partial D\cap \bar
B$. 
Any ball $B$ as above shall
be called {\em an
interior ball} for $D$ at $\hat x$. 
We furthermore say that $D$
satisfies  the {\em interior
ball condition} if the condition is satisfied at any point of its boundary.
We let $\frak r_D(\hat x)$ denote the maximal radius, that is less
than, or equal to $1$, of an interior ball
at the given $\hat x\in\partial D$.
We  say that $D$
satisfies  the {\em uniform interior
ball condition} if it satisfies the interior ball condition and 
\begin{equation}
\label{IBC}
\mbox{there exists $r_D\in(0,1)$ such that $\frak r_D(\hat x)\ge
  r_D,\mbox{for all } \hat
  x\in \partial D$.}
\end{equation}

\end{definition}

Suppose  $D$ satisfies the interior ball condition at $\hat x\in \partial D$
and $B(y,r)\subset D$, with $r>0$,  is an
interior ball at $\hat x$.   Any vector of the form
\begin{equation}
\label{bnr}
{\bn}:=\frac{1}{r}(\hat x-y)
\end{equation}
shall be called a {\em generalized unit outward normal vector} to $
\partial D$ at $\hat x$. Denote by 
${\frak n}(\hat x)$ the set of all such vectors at $\hat x\in \partial
D$.

\begin{definition}[Lower outward normal  derivative at a boundary point]
\label{def022605-20}
Suppose that  $D$ satisfies the interior ball condition at $\hat x\in
D$. Assume furthermore that  $f:\bbR^d \to\bbR$.
For any ${\bf n}\in {\frak n}(\hat x)$ define the {\em lower outward normal}  derivative at $\hat x$ as 
\begin{equation}
\label{062605-20}
\underline{\partial}_{\bn}f(\hat x):=\mathop{\liminf_{h\to0+}}_{\hat
x-h{\bn}\in D}\frac{f(\hat
x)-f(\hat
x-h{\bn})}{h}.
\end{equation}
\end{definition}
\begin{remark} 
{If $\partial D$ is $C^2$ smooth, then for any $\hat
x\in \partial D$ the set ${\frak n}(\hat x)$ is
a singleton and consists only of the unit
outward normal vector to $\partial D$ at $\hat x$.
In addition, if  $f:\bbR^d \to\bbR$ possesses a
derivative at $\hat x\in \partial D$, then
$\underline{\partial}_{\bn}f={\partial}_{\bn}f$ - the outward normal derivative.}
\end{remark}

In some cases, see e.g. Theorem \ref{cor.d21.20.1} below, we shall assume
that the domain satisfies the exterior ball condition in the following sense.
\begin{definition}
\label{def3.1z}
  {For a given domain $D$ 
and $\hat x\in\partial D$, we say that the {\em exterior
ball condition} is satisfied at $\hat x$, if there exists $B$ -  an open
ball of a positive
radius - contained in $D^c$ and such that $\hat x\in \partial D\cap \bar
B$. Any ball $B$ as above shall
be called {\em an
exterior ball} for $D$ at $\hat x$.  We furthermore say that $D$
satisfies  the {\em exterior
ball condition} if the condition is satisfied at any point of its boundary.}
\end{definition}

\section{Maximum principles and quantitative Hopf Lemmas}

\label{sec3}

The present section is devoted to the formulation of the first part of
the set of our main results. Their proofs are presented throughout Section \ref{sec5}.

\subsection{Maximum principles and the Hopf lemma}

Our first result is a form of a weak maximum principle. 
\begin{proposition}{[Weak maximum principle]}
\label{prop012804-20}
Suppose that the function $c$ 
  satisfies  A3). If $u$  is a subsolution to
\eqref{eq3.1a} on a  {bounded, open set  $D\subset \BR^d$} then (cf \eqref{cSD})
\begin{equation}
\label{032804-20}
\sup_{x\in D}u(x)\le  {\sup_{y\in(\mathcal S(D)\setminus D)\cup\partial
D}u^+(y) } .
\end{equation}
Moreover, if $c\equiv 0$, then $u^+$ in the right-hand side can be replaced by $u$.
\end{proposition}
The proof of the proposition is shown in Section \ref{sec5.1b}.

 {
\begin{remark}
Observe that the right-hand side of \eqref{032804-20} reduces to
$\sup_{y\in D^c}u^+(y)$ in the case $\mathcal S(D)=\BR^d$ 
and to $\sup_{y\in\partial D}u^+(y)$ when $\mathcal S(D)=\overline D$.
\end{remark}}

 {
\begin{theorem}[Bony maximum principle]
\label{thm:main3a}
Let $D$ be an open subset of $\BR^d$ and $ {\mathcal S(D)\subset \overline D}$. 
Suppose furthermore  that  $u\in W^{2,p}_{loc}(D)\cap C_b(\BR^d)$ for some $p>d$ and $\hat x\in
D$ is its  maximum point in $D$.
Then, cf \eqref{essinf},
\begin{equation}
\label{010104-20aabb}
\liminf_{x\rightarrow \hat x} {\rm ess} \,Au(x)\le 0.
\end{equation}
\end{theorem}
The proof of the theorem is presented in Section \ref{sec-bony}.
}

In an analogy to the case of elliptic differential operators, we have the strong maximum
principle and Hopf lemma for a subsolution of \eqref{eq3.1a}.  The
formulation of the latter,  given below, is of a bit more
''quantitative'' nature
  than   usually encountered in the literature. To state it we
  introduce some notation. For $r>0$ we
let 
\begin{equation}
\label{Da}
D_r:=[x\in D: \mbox{dist}(x,\partial D)>r].
\end{equation}
{We also  let ${\cal O}_D:=[x\in\bbR^d:\,{\rm dist}(x,D)<1]$.}
\begin{theorem}
\label{strong-max}
Suppose that  the function $c$  
  satisfies  A3).
 \begin{enumerate}
  \item[1)] (The strong maximum principle)
If $u$ is a subsolution to
  \eqref{eq3.1a} on a domain $D$ that  attains its maximum $M\ge0$ {over $\mathcal S(D)\cup \overline D$}
at $\hat x\in   D$, then $u\equiv M$ in $D$. 
\item[2)] (The  Hopf lemma)
Assume that a domain $D$   {  is  bounded and} satisfies the uniform interior ball
condition, with constant $r_D\in(0,1]$. Then, there exist  $r\in
(0,r_D/2], a>0$ depending only on $M_A$ and  {$\ga_{\bar {\cal O}_D}$} such that
for any non-constant  subsolution  $u$ to \eqref{eq3.1a} that attains
its non-negative  maximum  {over $\mathcal S(D)\cup\overline D$}  at
$\hat x\in\partial D$, and  any ${\bf n}\in {\frak n}(\hat x)$ at
$\hat x\in   \partial D$  we have
\begin{equation}
\label{010104-20a}
\underline{\partial}_{{\bn}}u(\hat x)\ge a {\inf_{x\in D_{r}}}(u(\hat x)-u(x))>0.
\end{equation}
\end{enumerate}
\end{theorem}
The proof of the theorem is given in Section \ref{sec4.2}. 


\medskip

\subsection{Quantitative Hopf lemmas}

Suppose that $u$ is a subsolution of \eqref{eq3.1a}.
If $c$ satisfies either the hypothesis A3'), or A3''), then we can
formulate a lower bound on the outward normal partial derivative of
$u$ at the maximum
point in terms of its maximal value. 
\begin{theorem}[Quantitative Hopf Lemma I.A]
\label{thm:main1a}
Suppose that function $c$ 
  satisfies hypothesis A3''). In addition, assume that a domain
  $D$ {is bounded} and satisfies the  uniform interior
ball condition. Then, there exists $a>0$ that depends
  only on $r_D$ of \eqref{IBC},  {ellipticity constant $\ga_{\bar {\cal O}_D}$},
  the lower bound  $\underline c$ and $M_A+\|c\|_\infty$
(cf \eqref{MA}) such that for any
$u$ - a subsolution to  \eqref{eq3.1a} - that attains its maximum  {over $\mathcal S(D)\cup\overline D$}
at $\hat x\in \partial D$ and is not identically equal to $0$  we have
\begin{equation}
\label{010104-20aa}
\underline{\partial}_{{\bn}}u(\hat x)>  a u(\hat x).
\end{equation}
\end{theorem}
The proof of the theorem is presented in Section \ref{sec6.3a}.

\bigskip

A version of the  result in the case when $c$ is only assumed to be
non-negative and not
equal to $0$ a.e. can be stated as follows.
\begin{theorem}[Quantitative Hopf lemma I.B]
\label{thm:main1}
Suppose that function $c$ 
  satisfies A3'). In addition, assume that a {bounded} domain $D$ satisfies
  the  interior ball condition. Then, there exist
 a constant $r_0>0$,    depending 
  only on $D$, ${\ga_{\bar {\cal O}_D}}$ and
  $M_A+\|c\|_\infty$ (cf \eqref{MA}),
and 
a nondecreasing function $\rho_{c,A,D}:(0,+\infty)\to(0,+\infty)$  given by formula
\eqref{eq.uewo.dism} below and
depending only on $A,c$ and $D$ such that for 
$u$ and 
$\hat x\in \partial D$ as in Theorem \ref{thm:main1a} we have
\begin{equation}
\label{010104-20}
\underline{\partial}_{{\bn}}u(\hat x)>  a \rho_{c,A,D}(r_0\wedge \frak r(\hat x)) u(\hat x).
\end{equation}
Here $\frak r(\cdot)$ is as in Definition \ref{def2.3}.
 \end{theorem}
The proof of the theorem is presented in Section \ref{sec6.3}.

\begin{remark}
\label{rmk3.7}
{At this point we wish to emphasize  an important distinction between
the conclusions of Theorems  \ref{thm:main1a}  and
\ref{thm:main1}, see also Lemmas \ref{lm010204-20} and
\ref{lm010204-20z} below. Given a
fixed region $D$, the
constant $a$, appearing in the statements of both results,   can be
chosen uniformly for the class of operators of the form $A-c$, so long
as their appropriate coefficients satisfy  the uniform ellipticity
condition and   respective  lower and upper
bounds with the same
constants. On the other hand,
the   function $\rho_{c,A,D}(\cdot)$, appearing in Theorem
\ref{thm:main1} and defined in \eqref{eq.uewo.dism},  depends on the particular
operator $-A+c$  {through its Green function} (see formula \eqref{eq.wcd1}) and we do not claim that it can be
chosen uniformly with respect to
the   numerical
bounds on the coefficients. It is the price we  pay
for weakening of the assumption on $c$.
In Theorem \ref{thm:main1a} we have assumed that $c$ is bounded away from
zero (hypothesis A3'')), while in Theorem \ref{thm:main1} it sufficed only to
suppose  that $c$ is merely non-trivial (hypothesis  A3')).}
\end{remark}

\subsection{Harnack property and the Hopf lemma}
For operators $A$ having the {\em Harnack property}  and $A$-harmonic
functions we can replace the right hand of \eqref{010104-20a} by a
constant multiplied by $u(\hat x)-u(x_0)$ for a fixed $x_0\in D$.
To rigorously formulate the result we start with the notion of an $A$-harmonic function.
{\begin{definition}
Suppose that $p>d$ and $D\subset \bbR^d$ is open. 
We say that a function $u\in C_b(\BR^d)\cap W^{2,p}_{loc}(D)$ is
$A$-{\em harmonic} in $D$ iff 
 $Au=0$ a.e. in $D$.  
\end{definition}

 {\begin{definition}
\label{hnk}
We say that $A$ has the {\em Harnack property} in $D$  iff for any $r\in
(0,1]$ there exists $C>0$ such that for any $y_0\in D$ satisfying
$B(y_0,r)\subset D$
and a non-negative   $u\in C_b(\BR^d)$,  which is $A$-harmonic in $B(y_0,r)$, we have
\begin{equation}
\label{harnack}
u(x)\le Cu(y),\quad x,y\in B(y_0,r/2).
\end{equation}
\end{definition}
\begin{remark}
In general, the Harnack property need not hold for $A$.   {In fact}, some additional hypotheses on the L\'evy kernel $N$ are needed.
Some sufficient conditions  on  the kernel guaranteeing the Harnack
property of $A$ can be found in  \cite[Theorem 2.4, p. 25]{Foondun}.
\end{remark}

The following result is a straighforward conclusion from Theorem
\ref{strong-max} and the Harnack
property of $A$.
\begin{corollary}
\label{thm:harnack}
In addition to the assumptions made in Theorem \ref{strong-max}
suppose that $A$ has the Harnack
property  {in a bounded domain $D$}. Then, for any  $x_0\in D$  there
exists $a>0$ depending only on $M_A$, $\ga_{\bar{\cal O}_ D}$, ${\rm
  diam}\,D$
and  {$\delta_{D}(x_0)( ={\rm dist}(x_0,\partial D))$} such that
for any   {non-constant}  $u$, which is $A$-harmonic in $D$ and attains its maximum  {in} $\mathcal S(D)\cup\overline D$ at $\hat x\in\partial D$,
we have
\begin{equation}
\label{010104-20b}
\underline\partial_{\mathbf n}u(\hat x)\ge a(u(\hat x)-u(x_0))>0.
\end{equation}
\end{corollary}
\proof
{By Theorem \ref{strong-max} we have estimate \eqref{010104-20a} that
holds for some positive constants $a$ and $r$ depending only on  
$M_A$ and $\ga_{\bar{\cal O}_ D}$. Let
\begin{equation}
\label{rp}
r':=\min\left\{\frac{1}{2}\delta_{D}(x_0),r,1\right\}.
\end{equation}
 Then \eqref{010104-20a} holds also with $D_r$ replaced by $D_{r'}$. Clearly $D$ can be contained in a $d$-dimensional box
of  sidelength $ {\rm
  diam}\,D$. 
  We can cover the box with 
  \[
  N=\Big\{\big[4{\rm
  diam}\,D \sqrt{d}/r'\big]+1\Big\}^{d}
  \]
 balls of radius $r'/4$,
   where $[
\cdot]$ is the integer part of a number (the floor function). 
Let ${\cal G}$ be the family of these  balls which have a non-empty intersection with $\bar
D_{r'}$. Obviously $2B\subset D$ for each $B\in \mathcal G$.  Let $x\in D_{r'}$. 
Since $D$ is connected there exist $n\ge 1$ and $B_1,\dots,B_n\in\mathcal G$
such that $x_0\in B_1, x\in B_n$ and $B_i\cap B_{i+1}\neq\emptyset$.
Let $y_i\in B_i\cap B_{i+1},\, i=1,\dots,n-1$.
After an application of the Harnack
inequality \eqref{harnack}   to the  non-negative
harmonic function $u(\hat x)-u(x)$ on each $B_i,\, i=1,\dots,n$, we conclude that
\[
\inf_{y\in B_{i+1}}(u(\hat x)-u(y))\ge C(r')\inf_{y\in B_{i}}(u(\hat x)-u(y)),\quad i=1,\dots,n-1,
\]
and 
\[
u(\hat x)-u(x)\ge C(r')\inf_{y\in B_{n}}(u(\hat x)-u(y)),\quad \inf_{y\in B_{1}}(u(\hat x)-u(y))\ge C(r')(u(\hat x)-u(x_0)).
\]
Here $1/C(r')$ is the constant appearing in \eqref{harnack} that
corresponds to $r'$ given in \eqref{rp}.
Thus
\[
u(\hat x)-u(x)\ge \min_{i=1,\dots,n}\inf_{y\in B_i}\big(u(\hat x)-u(y)\big)\ge [C(r')\wedge 1]^N(u(\hat x)-u(x_0)).
\] 
Taking the infimum over $x\in D_{r'}$ in the above inequality  and
using   \eqref{010104-20a} we conclude \eqref{010104-20b}.}
\qed

\section{Probabilistic preliminaries}

\label{sec4}

Throughout this section we shall assume that the hypotheses A1) -- A3)
hold.

 \subsection{On the martingale problem associated with operator $A$} 

\label{SMM}

 Let ${\cal D}$ be the
space consisting of all 
functions $\om: [0,+\infty)\to\BR^d$, that are right continuous and
possess the left limits for all  {$t>0$} (c\'{a}dlags), equipped with
the Skorochod topology, see e.g. Section 12 of \cite{bil}. 
Define the canonical process $X_t(\om):=\om(t)$, $\om \in {\cal D}$
and its natural filtration $({\cal F}_t)$, with ${\cal
  F}_t:=\si\left(X_s,\,0\le s\le t\right)$.

\begin{definition}[A solution of the martingale problem associated
  with the operator $A$]
Suppose that $\mu$ is a Borel probability measure on $\bbR^d$.  A
Borel  probability 
measure $P_{\mu}$ on $ {\cal D} $ is called a {\em solution of the martingale problem}  associated
  with $A$ with the initial distribution $\mu$, cf \cite{Stroock}, if
\begin{itemize}
\item[i)] $P_{\mu}[X_0\in Z]=\mu[Z]$ for any Borel measurable
 $Z\subset \bbR^d$.
\item[ii)] For every $f\in C_b^2(\BR^d)$
- a $C^2$-smooth function bounded with its two derivatives on $\R^d$ - the process
\begin{equation}
\label{Mtf}
M_t[f]:= f(X_t)-f(X_0)-\int_0^tAf(X_r)\,dr,\, t\ge 0
\end{equation}
is a  (c\'{a}dlag)  martingale under measure $P_\mu$ with respect to  {natural filtration
$({\cal F}_t)_{t\ge0}$}. 
As usual we
write $P_x:=P_{\delta_x}$,  $x\in\bbR^d$ and say that $x$ is the
initial condition. The expectations   with respect to $P_{\mu}$ and
$P_x$ shall be denoted by $\E_{\mu}$ and $\E_x$, respectively.
\end{itemize}
\end{definition}

\begin{definition}[A strong Markovian solution of the martingale problem]
We say that a family of Borel probability measures
  $(P_x)_{x\in\bbR^d}$   on
  ${\cal D} $ is  a strong Markovian solution to the martingale problem  associated with $A$
if:
\begin{itemize}
\item[i)] each $P_x$ is a solution of the martingale problem
  associated with $A$, corresponding to the
initial condition at $x$,

\item[ii)] the canonical  process $(X_t)$ is 
  strongly Markovian with respect to the natural filtration $({\cal
    F}_t)$ and the family $(P_x)_{x\in\bbR^d}$,
\item[iii)] the mapping $x \to P_x[C]$ is measurable for any Borel
  $C\subset{\cal D}$,
\item[iv)] for any Borel probability measure $\mu$ on $\bbR^d$  the probability measure
 $$
P_{\mu}(\cdot) :=\int_{\bbR^d}
P_x (\cdot) \mu(dx)
$$
is a solution to the martingale problem associated with $A$ with the
initial distribution $\mu$.
\end{itemize}
\end{definition}

\bigskip

The following result concerning the existence of a strong Markov
solution for a martingale problem has been  proven in \cite{AP}, see also
\cite{LM}.
\begin{theorem}
Suppose that conditions \eqref{la-K} and \eqref{MA} are
satisfied. Then, 
the martingale problem associated with the operator $A$ admits a strong Markovian solution.
\end{theorem}

\subsection{Analytic description of the canonical process and weak subsolution to (\ref{eq3.1a})}

\subsubsection{Exit time, transition probability semigroup,
  resolvent operator} 

\label{ETS}

For a given domain $D$ define 
the exit time $\tau_D:{\cal D}\to[0,+\infty]$ of the canonical process $(X_t)_{t\ge0}$ from $D$ as
\begin{equation}
\label{tau-D}
\tau_D:=\inf[t>0:\,X_t\not\in D].
\end{equation}
It is a stopping time,
i.e. for any
$t\ge0$ we have $[\tau_D\le t]\in {\cal F}_t$, see Theorem I.10.7, p. 54 of
\cite{bg} and Theorem IV.3.12, p. 181 of \cite{ethier-kurtz}. 
\begin{proposition}
\label{lm010104-20}
Suppose that $D$ is a bounded domain. Then, 
\begin{equation}
\label{012904-20}
\sup_{x\in D}\E_x\tau_D<+\infty.
\end{equation}
\end{proposition}
\begin{proof}
See e.g.  \cite[Lemma 4]{Foondun}.
\end{proof}

We let
\begin{equation}
\label{PDt}
P^D_tf(x):=\E_x\left[f(X_t)1_{[t<\tau_D]}\right]\quad t\ge 0,\quad f\in B_b(D).
\end{equation}
This is a probabilistic formula for the semigroup generated by
operator (\ref{011703-20}) on $D$ with the null exterior condition.

For a positive $f\in B(D)$ and $\al\ge 0$ we let
\begin{equation}
\label{RDt}
R^D_\alpha f(x):=\int_0^\infty e^{-\alpha t}P^D_tf(x)\,dt =\E_x\left[\int_0^{\tau_D}e^{-\alpha r}f(X_r)\,dr\right],\quad x\in D.
\end{equation}
Set $R^D:=R^D_0$. We have $R^D_\alpha f(x)<+\infty$, $x\in D$ for
$f\in B_b(D)$ and $\alpha>0$.

\subsubsection{The gauge function for $c(\cdot)$ and domain $D$}

\label{sec4.4}

\begin{definition}[The gauge function for $c(\cdot)$ and domain $D$]
\label{df4.2}
The function
\begin{equation}
\label{eq2.pvf}
v_{c,D}(x)=\E_xe_{c}(\tau_D),\quad x\in \BR^d,
\end{equation}
where, {the random functional}, 
\begin{equation}
\label{ecD}
e_{c}(t):=e^{-\int_0^{t}c(X_r)\,dr},\quad t\ge0,
\end{equation}
is called the {\em gauge function} corresponding to $c(\cdot)$ and domain $D$, cf Section 4.3 of \cite{cz}.
\end{definition}
Let us denote
\begin{equation}
\label{wcD}
w_{c,D}(x):=1-v_{c,D}(x),\quad x\in \BR^d.
\end{equation}
Obviously  $w_{c,D}(x) \ge0$, $x\in \BR^d$. In fact, as it turns out,
under the assumptions made in the present paper $w_{c,D}(x)$
is strictly positive and continuous in $D$, see Lemma
\ref{lm010809-20} below.

\subsection{Feynman-Kac formula  for $(X_t)$}

\begin{proposition}
\label{prop010204-20} 
\begin{enumerate}
\item[1)]
 {If $u\in W^{2,p}_{loc}(D)$ and $p>d$, then 
 for any  {$V\Subset D$} and $t\ge 0$
 \begin{equation}
\label{010204-20dfg1c}
u(x)=\E_x\left[e_c(\tau_V\wedge t)u(X_{\tau_V\wedge t})\right]+\E_x\left[\int_0^{\tau_V\wedge t}e_c(s)(-Au+cu)(X_s)\,ds\right],\quad
x\in  V.
\end{equation}
 Moreover, \eqref{010204-20dfg1c} holds also with $\tau_V$ in place of $\tau_V\wedge t$.}
 \item[2)]
 If $u$ is a subsolution of  \eqref{eq3.1a}, then
\begin{equation}
\label{010204-20dfg}
u(x)=\E_x\left[e_c(\tau_D\wedge t)u(X_{\tau_D\wedge t})\right]+\E_x\left[\int_0^{\tau_D\wedge t}e_c(s)(-Au+cu)(X_s)\,ds\right],\quad
x\in  D.
\end{equation}
for any $t\ge 0$. Moreover, {if $D$ is bounded, then} \eqref{010204-20dfg} holds also with $\tau_D$ in place of $\tau_D\wedge t$.
\end{enumerate}
\end{proposition}
The proof of this result is given in Section \ref{sec3.1} below.

\subsection{Principal eigenvalue and eigenfunction associated with
  $A$} 

 The following result is the version of the Krein-Rutman theorem for the
operators considered in Section \ref{sec2.1}.
\begin{theorem}
\label{thm020809-20}
Suppose that the domain $D$ is bounded. Then, there exists a unique pair $(\varphi_D, \lambda_D)$, where
$\varphi_D:D\to(0,+\infty)$ -  a strictly positive continuous function
- and  $\lambda_D>0$
are such that {$\int_D\varphi_Ddx=1$ and}
\begin{equation}
\label{eigen}
e^{-\lambda_D t}\varphi_D(x)=P^D_t\varphi_D(x),\quad t\ge 0,\,x\in D.
\end{equation}
\end{theorem}

The proof of the theorem is presented in Section \ref{sec5.1}

\begin{definition}
The function $\varphi_D$ and $\lambda_D$  are called respectively:  the {\em principal
eigenfunction} and its associate {\em principal eigenvalue} for the operator $A$, with the zero Dirichlet exterior condition {on $D^c$}, see
e.g. \cite{BNV}.
\end{definition}
\bigskip

\section{Generalizations of the quantitative Hopf lemmas and weak Harnack estimates}

\label{sec5z}

Throughout this section we shall always assume that   domain $D$ is bounded. 

\subsection{Some generalizations of Theorems \ref{thm:main1a} and \ref{thm:main1}}

Below, we   formulate some generalizations of
estimates \eqref{010104-20aa} and \eqref{010104-20}. In general, we shall no
longer assume that the boundary $\partial D$ is regular enough so that
the outward normal, appearing in the left hand sides of the
aforementioned bounds,  can be defined. It shall be  replaced by the ratio
$[u(\hat x)-u(x)]/\psi(x)$, $x\in D$. Here $\hat x\in\partial D$ and $\psi:D\to(0,+\infty)$ are the
maximum point of the subsolution $u$ and some suitably defined function (e.g. the
principal eigenfunction associated with the operator), respectively.  In case $ \psi(x)\approx \delta_D(x)( ={\rm
  dist}(x,D^c))$, when $x\to\hat x$ (which may happen for sufficiently
regular domains) these type of estimates  lead to \eqref{010104-20aa} and \eqref{010104-20}.

\begin{theorem}[{Quantitative Hopf lemma {II.A}}]
\label{pee1}
Suppose that the assumption  A3)
holds.
Then, there exist a continuous function $\psi:D\to(0,+\infty)$ and a Borel function 
$\chi:D\to(0,+\infty)$, such that   for  any 
subsolution $u(\cdot)$ of \eqref{eq3.1a} satisfying  {
$$u(\hat
x)=\max_{x\in  \mathcal S(D)\cup\overline D}u(x)\ge 0\quad\text{for
  some}  \quad\hat x\in\partial D
$$
we have
\begin{equation}
\label{012804-20-1}
u(\hat x)-u(x)\ge \frac{\underline
  c\varphi_D(x) u(\hat
x)}{2e
  \|\varphi_D\|_\infty (\la_D+\underline c)}+\psi(x)\int_D(Au-cu)\chi\,dx,\quad x\in D.
\end{equation}}
\end{theorem}
The result is
proved  in Section \ref{sec7.2.2}.

\medskip

Instead of the second term in the right hand side of
\eqref{012804-20-1} we can use a
term depending on the principal eigenfunction, but then we have to
replace the mean of $Au-cu$, with respect to some measure, by its essential infimum.
\begin{theorem}[{Quantitative Hopf lemma  {II.B}}]
\label{pee}
Suppose that 
$u(\cdot)$ and $\hat
x$ are as in Theorem \ref{pee1}. Then, under the assumptions of  the
above theorem 
we have (cf \eqref{c-c})
\begin{equation}
\label{012804-20}
u(\hat x)-u(x)\ge \frac{\varphi_D(x)}{2e
  \|\varphi_D\|_\infty }\left[\frac{\underline
  cu(\hat
x)}{\la_D+\underline c}+\frac{{\rm essinf}_D(Au-cu)}{\la_D+\|c\|_\infty}\right],\quad x\in D.
\end{equation}
\end{theorem}
The result is
proved  in Section \ref{sec7.2.1}.

\begin{remark}
We stress the fact that in both Theorems \ref{pee1} and
\ref{pee} no assumption    
on $\partial D$ is made, except that it
is a boundary of a bounded domain. 
\end{remark}

\subsection{Remarks on Theorem \ref{pee1}}

\label{sec5.2z}
It is interesting to compare Theorem \ref{pee1} with Lemma 3.2 of
\cite{BC}, where a similar estimate is obtained for solutions
  of
the homogeneous Dirichlet boundary problem for the Poisson  equation, i.e.
 \begin{equation}
\label{poisson}
-\Delta u=f\ge 0\quad\text{ and }\quad u(x)=0,\,x\in\partial D.
\end{equation}
 According to \cite{BC}  if $f\in L^\infty$,
$\partial D$ is smooth, then there exists $a>0$, depending only on $D$, such that for any $u$
satisfying \eqref{poisson} we have
\begin{equation}
\label{011210-20}
 u(x)\ge \psi(x)\int_D f\chi\,dx,\quad x\in D,
\end{equation}
where
\begin{equation}
\label{psi-chi}
\chi(x)=\delta_D(x) \quad \text{ and }\quad\psi(x)=a\delta_D(x).
\end{equation}
{Note  that in Theorem \ref{pee1}
no assumptions on smoothness of $\partial D$ is made}. In addition, the operator $A$ is no
longer just the laplacian. Instead, it suffices that 
hypotheses A1) and A2) hold.  Obviously, in such a
generality  one cannot hope to
show   an analogue of \eqref{012804-20-1}, with the
  functions $\psi$, $\chi$ given by \eqref{psi-chi}, see Remark \ref{rmk5.5} below. The estimate  however holds, provided
both the domain and the
coefficients of $A$ are sufficiently regular. We have the following.
\begin{theorem}
\label{cor.d21.20.1}
Assume that $D$   satisfies  both the uniform interior ball and  exterior ball
conditions.  Then, there exists $a>0$ such that   \eqref{012804-20-1}
holds with $\psi$ replaced by $a\delta_D$ {and $\chi$ unchanged}.
\end{theorem}
The proof of the {theorem} can be found in Section \ref{sec7.2.2a}.

\medskip

\begin{remark}
\label{rmk5.5}
An example of a domain that satisfies both the interior and exterior ball conditions, but not in the uniform sense, see \cite[Theorem 4]{BeSw}, shows that 
even in the case when $A$ is the laplacian one cannot use
$\psi=a\delta_D$ in \eqref{012804-20-1} (thus also in
\eqref{011210-20}, with $\psi$ as above) for some $a>0$. On the other
hand, the {classical} Hopf lemma, {as formulated in \cite[p. 82]{BeSw},}   is then satisfied.
\end{remark}

\medskip

We can further strengthen Theorem  \ref{cor.d21.20.1}, allowing $\psi$ and $\chi$ as in \eqref{psi-chi}, if some additional
  assumptions on $A$ are made. Consider  the operator $\hat A$ that is the formal adjoint
  (w.r.t. the Lebesgue measure) to $A$, i.e. 
\begin{equation}
\label{eq.deo.1}
\int_{\BR^d} vAu\,dx=\int_{\BR^d} u \hat A v\,dx,\quad u,v\in C^2_c(\BR^d).
\end{equation}
Assume that it is of the form
  $\hat A_0-\hat c$, where 
\begin{align}
\label{eqi.6.30.09}
\nonumber &
\hat A_0u(x)=\frac12\sum_{i,j=1}^d\hat q_{i,j}(x)\partial_{x_i,x_j}^2u(x)+\sum_{i=1}^d\hat b_{i}(x)\partial_{x_i}u(x)
\\
& 
  +\int_{\BR^d}\left(u(x+y)-u(x)-\sum_{i=1}^d\frac{y_i\partial_{x_i}u(x)}{1+|y|^2}\right)\hat
  N(x,dy),\quad u\in C^2(D)\cap C_b(\bbR^d),
\end{align}
 with the coefficients $\hat q_{i,j}$,  $\hat b_{i}$, $\hat N$
satisfying A1), A2) and  $\hat c\in B_b(\bbR^d)$.
Define $\widehat {\cal S}(D)$ using \eqref{cSD} for $\hat N$.

 Let
$(\hat P_x)$ be a strong Markovian solution to the martingale problem
corresponding to $\hat A_0$. Then, by the strong Markov property
\begin{equation}
\label{hPDt}
\hat P_t^{\hat c,D}f(x):=\hat{\bbE}_x\left[f(X_t)e_{\hat c}(t)\mathbf1_{\{t<
  \tau_D\}}\right],\quad x\in D,\,t\ge0,\,f\in B_b(D)
\end{equation}
 defines a semigroup of operators
$(\hat P_t^{\hat c,D})_{t\ge0}$ on $B_b(D)$. Let
 $\hat R^{\hat c,D}_\al$ be the respective resolvent operator,
 defined by
\eqref{RDt}, where $P_t^{D}$ is replaced by $\hat P_t^{\hat
  c,D}$. It can be defined for all $\al\ge \|\hat c^+\|_\infty$. Assume that
\begin{equation}
\label{res-dual}
\int_D \hat R^{\hat c,D}_\alpha f g\,dx=\int_Df R^D_\alpha g\,dx,\quad f,g\in B^+(D),\, \al\ge \|\hat c^+\|_\infty.
\end{equation}
The following result
can be inferred from {the quantitative Hopf lemma expressed in} Theorem \ref{pee1}. Its proof is given in Section \ref{sec7.2.2a}.
\begin{theorem}
\label{cor.d21.20.1a}
{Assume that $D$   satisfies  both the uniform interior and  exterior ball
conditions} and $\hat A$ is as described above. 
Then, the conclusion of Theorem \ref{pee1} holds  with  $\psi$ and $\chi$
as in \eqref{psi-chi}.
\end{theorem}

In order to use Theorem \ref{cor.d21.20.1a} we need to verify that
the resolvent $\hat R^{\hat c,D}_\alpha$   satisfies
\eqref{res-dual}. The equality holds, provided that both $\partial D$ and
coefficients of $\hat A$  are sufficiently
regular.  Before we formulate some sufficient condition let us first
recall the notion of a function of vanishing mean oscillation.

\begin{definition}[Functions of vanishing mean oscillation]
\label{df5.6}
Given a function $f\in L^1_{\rm loc}(\bbR^d)$ and $r>0$ let us define
$$
\eta(r):=\sup_{B\in {\frak B}_r}\frac{1}{m_d(B)}\int_{B}\left|f(x)-f_B\right|dx.
$$
Here ${\frak B}_r$ denotes the family of all balls of radius less
than, or equal to $r$ and \linebreak $f_B:=(1/m_d(B))\int_{B}f(x) dx$.
We say that a function is of vanishing mean oscillation (VMO) if
$\sup_{r>0}\eta(r)<+\infty$, i.e. it is of bounded mean oscillation
(BMO), and $\lim_{r\to0}\eta(r)=0$. Denote by $VMO$ the class of all
 such functions.
\end{definition}

The following result is shown in Section \ref{sec-res-dual}.
\begin{theorem}
\label{thm013009-20}
Suppose that $\partial D$ is of $C^2$ class, the coefficients of both $A$ and $\hat A_0$ satisfy
hypotheses A1) and A2) and $\hat c\in B_b(\bbR^d)$.   
Furthermore, assume that    $q_{i,j}$,  $\hat q_{i,j}$ belong to
$VMO$ for all $i,j=1,\ldots,d$ and  $\mathcal S(D)\cup
\widehat{\mathcal S}(D)\subset  {\overline D}$. Then
\eqref{res-dual} holds {and, as a result, the conclusion of Theorem
  \ref{pee1} reamin valid with  $\psi$ and $\chi$
as in \eqref{psi-chi}.}
\end{theorem}
\medskip


\subsection{Weak Harnack inequality}

We end this section with the result that can be interpreted as a weak
form of the Harnack inequality.
\begin{theorem}
\label{thm:main4a}{
Let $D$ be a bounded domain in $\BR^d$ and $c$ satisfies hypothesis A3).
Then, there exists a bounded Borel measurable function $\chi:D\to(0,+\infty)$, such that
for any $V {\Subset} D$ we can find a constant $C>0$   such that
for any  $u$ - a  non-negative supersolution to \eqref{eq3.1a} -  we have
\begin{equation}
\label{830104-20aabb}
\inf_{x\in V} u(x)\ge C \int_Vu\chi\,dx.
\end{equation}}
\end{theorem}
The proof of the result is shown in Section \ref{prf-thm5.4}.
\medskip

A direct consequence of part 2) of Theorem \ref{strong-max} and Theorem
\ref{thm:main4a} is the following.
\begin{corollary}
\label{cor012909-20}
Under the assumptions of part 2) of Theorem \ref{strong-max} and Theorem
\ref{thm:main4a} for any $r>0$ there exist $C>0$  such that
 any non-constant  subsolution  $u$ to \eqref{eq3.1a}, that attains
its non-negative  maximum  {over $\mathcal S(D)\cup\overline D$}  at
$\hat x\in\partial D$, and  any ${\bf n}\in {\frak n}(\hat x)$ at
$\hat x\in   \partial D$  satisfies (cf \eqref{Da})
\begin{equation}
\label{010104-20az}
\underline{\partial}_{{\bn}}u(\hat x)\ge C \int_{D_r}(u(\hat x)-u(x))\chi(x)dx>0.
\end{equation}
Here  $\chi$ is as in  the statement of Theorem \ref{thm:main4a}.
\end{corollary}

\section{Proofs of the Feynman-Kac formula and weak maximum principle}

\label{sec6}

\subsection{Proof of Proposition  \ref{prop010204-20} - the Feynman-Kac formula}
\label{sec3.1} 


Suppose that $u\in  W^{2,p}_{\rm loc}(D)\cap C_b(\BR^d)$ for
some $p>d$ is a
subsolution of  \eqref{eq3.1a}.   We can find
 $(u_n)\subset   C^2_c(\BR^d)$ such that  
{$\|u-u_n\|_{W^{2,p}(V)}\to0$} for any $V \Subset D$.
This in particular implies that
$u_n\rightarrow
u$ and $\nabla u_n\rightarrow \nabla u$ uniformly  on $V$,
and $\|D^2 u_n-D^2 u\|_{L^p(V)}\rightarrow 0$ for any $V \Subset D$,
as $n\to+\infty$. 
As a result for any such $V$, we have 
\begin{equation}
\label{042904-20}
\lim_{n\to+\infty}\|Au_n-Au\|_{L^p(V)}=0.
\end{equation}
By \cite[Theorem 4.3]{Stroock},
\begin{equation}
\label{042904-20v12}
M^n_t:= u_n(X_t)-u_n(x)-\int_0^tAu_n(X_r)\,dr,\, t\ge 0
\end{equation}
is a martingale under the measure $P_x$ for every $x\in \BR^d$.  By the It\^o formula, applied to
$e_c(t)u_n(X_t)$, for any 
$V \Subset D$
 we get
\begin{align*}
u_n(x)&=e_c(\tau_V\wedge t)u_n(X_{\tau_V\wedge
        t})+\int_0^{\tau_V\wedge t}e_c(r) (-Au_n+cu_n)(X_r)\,dr\\&
\quad-\int_0^{\tau_V\wedge t}e_c(r)\,dM^n_r,\quad t\ge0,\quad P_x\mbox{-a.s.}
\end{align*}
Hence, we conclude that
\begin{align}
\label{050909-20a}
u_n(x)=\E_x\Big[
  e_c(\tau_V\wedge t)u_n(X_{\tau_V\wedge t})\Big]+\E_x\left[\int_0^{\tau_V\wedge t}e_c(s)
  (-Au_n+cu_n)(X_s)\,ds\right],\quad x\in V.
\end{align}
Recall
the Krylov estimate
\begin{equation}
\label{anulova}
\sup_{x\in\bar
  V}\E_x\left[\int_0^{\tau_V}\left| f(X_t) \right|\,dt\right]\le
  C\|f\|_{L^d(V)},\quad f\in L^d(V)
\end{equation}
for some $C>0$, see \cite[Corollary p. 143]{AP}, or \cite[Thm
$III_{14}$, (53)]{LM}.
Letting $n\rightarrow \infty$ in \eqref{050909-20a} and using  \eqref{anulova} we get
\begin{align}
\label{050909-20}
u(x)=\E_x\Big[ e_c(\tau_V\wedge t)u(X_{\tau_V\wedge t})\Big]+\E_x\left[\int_0^{\tau_V\wedge t}e_c(s) (-Au+cu)(X_s)\,ds\right],\quad x\in V.
\end{align} 
Allowing then $t\rightarrow \infty$,  using  \eqref{anulova} and
the fact that $P_x(\tau_V<\infty)=1$, we conclude part 1) of the
proposition.

Concerning the proof of part 2). Let $\{D_n\}$ be an increasing sequence of relatively compact open
subsets of $D$ such that $D_n\subset D_{n+1}\Subset D$ and
$\bigcup_{n\ge 1} D_n=D$. Then the exit time $\tau_{D_n}$ increases to $
\tau_D$, as $n\to+\infty$. Taking $D_n$ in place of $V$ in \eqref{050909-20} and letting
$n\rightarrow \infty$, we get (using   quasi-left continuity of
the process, see Theorem IV.3.12, p. 181 of \cite{ethier-kurtz}) 
formula \eqref{010204-20dfg} for $x\in D$. Next, using the fact that
$Au-cu\ge0$ on $D$ (since $u$ is a subsolution of \eqref{eq3.1a}),
the fact that $P_x(\tau_D<\infty)=1$, and the
monotonne convergence theorem, we may let $t\rightarrow \infty$ in
\eqref{010204-20dfg}.\qed

 \subsection{Proof of Proposition  \ref{prop012804-20} - the weak maximum principle}



\label{sec5.1b}

 {As an immediate application of   the Feynman-Kac formula (Proposition
   \ref{prop010204-20}) we 
show the weak maximum principle formulated in Proposition \ref{prop012804-20}. 
Using formula \eqref{010204-20dfg}, remembering that $-Au+cu\le 0$, 
and letting $t\to+\infty$
we conclude that
\begin{align}
\label{eq.ad.wmpv71}
u(x)\le \mathbb E_x\left[e_c(\tau_D)u(X_{\tau_D})\right]\le \mathbb E_xu^+(X_{\tau_D}),\quad x\in D.
\end{align}
Since $D$ is bounded, cf
\eqref{012904-20},  we have  $\tau_D<+\infty$, $P_x$ a.s. This in particular
implies   that $X_{\tau_D}$ is a.s. well defined.}

Using the Ikeda-Watanabe formula, see \cite[Remark 2.46,
page 65]{BSW}, we obtain
\begin{equation}
\label{010793-20dfg1v}
\begin{split}
&\mathbb E_{x_0}\left[\sum_{0<s\le\tau_D}\mathbf{1}_{D}(X_{s-})\mathbf{1}_{\mathcal
    S(D)^c\setminus D}(X_s)\right]\\
&
=\mathbb
E_{x_0}\left[\int_{\BR^d}\int_0^{\tau_D}\mathbf{1}_{D}(X_{s-})\mathbf{1}_{\mathcal
    S(D)^c\setminus D}
(y)N(X_{s-},dy-X_{s-})\,ds\right].
\end{split}
\end{equation}
From the definition of $\mathcal S(D)$, cf \eqref{cSD}, we conclude
that the right-hand side of the above equation vanishes. Thus,
$X_{\tau_D}\in \mathcal S(D)\setminus D$, if
$X_{\tau_D}\not=X_{\tau_D-}$, or $X_{\tau_D}\in \partial D$, if
otherwise. Summarizing, we have shown that
\[
P_x(X_{\tau_D}\in (\mathcal S(D)\setminus D)\cup \partial D)=1,\quad x\in D
\]
and estimate \eqref{032804-20} follows. When $c\equiv 0$ we get \eqref{eq.ad.wmpv71}, with $u^+$ replaced by $u$.}\qed

\section{Properties of the resolvent  - proof of Theorem \ref{thm020809-20}}

\label{sec5.1}


Suppose that $D$ is a bounded domain.
Recall that $R^D_\al f$ is defined in \eqref{RDt} for all
$\al\ge0$ and $f\in B^+(D)$.
Using condition \eqref{012904-20} we infer that the operator 
$R^D_\al: L^\infty(D)\to L^\infty(D)$ is bounded for any $\al\ge 0$.
The operators  $(R^D_\al )_{\al\ge 0}$ satisfy the resolvent
identity
\begin{equation}
\label{res-id}
R^D_\al-R^D_\beta=(\beta-\al)R^D_\al R^D_\beta,\quad \al,\beta\ge 0.
\end{equation}
Thanks to assumptions A1)-A2) there exists $C>0$ such that 
\begin{equation}
\label{R-bdd}
\sup_{x\in D}\E_x\left[\int_0^{\tau_D}|f(X_t)|dt\right]\le
C\|f\|_{L^d(D)},\quad f\in L^d(D).
\end{equation}
see (53) of \cite{LM}, or Corollary 2 of \cite{mikulevicius}.
The estimate  implies that in fact $R^D_\al$ extends
to the operator $R^D_\al: L^d(D)\to L^\infty(D)$ for each $\al\ge0$.
This in turn allows us to conclude that 
for any $\al\ge 0$ there exists $r^D_\al:D\times D\to [0,+\infty)$ such that
\begin{equation}
\label{020205-20}
R^D_\al f(x)=\int_{D}r^D_\al(x,y)f(y)\,dy,\quad x\in D
\end{equation}
and
\begin{equation}
\label{020205-20a}
\mathop{\rm
  esssup}_{x\in D}\|r^D_\al(x,\cdot)\|_{L^{d/(d-1)}(D)}\le
\|R_\al\|_{L^d\to L^\infty}.
\end{equation}
The above implies  that $R^D_\al: L^d(D)\to L^d(D)$ is
compact for any $\al\ge 0$. This can be easily seen, as from any
bounded sequence $(f_n) \subset L^d(D)$ we can select a subsequence $(f_{n_k})$ 
weakly converging to some $f$. Using \eqref{020205-20} we conclude that 
$\lim_{k\to+\infty}R^D_\al f_{n_k}(x)=R^D_\alpha f(x)$ for a.e. $x\in D$. Since
$\left(\|R^D_\al f_{n_k}\|_{L^\infty}\right)$  is bounded, this implies the
strong  convergence of $(R^D_\al f_{n_k})$ to $R^D_\al f$ in $L^d(D)$.

By the support theorem of  \cite[Section 4]{Foondun}, or by \cite[Corollary 2]{mikulevicius2}  for any $x\in D$, 
\begin{equation}
\label{eq.uewo.st}
r^D_\alpha(x,\cdot)>0, \quad m\mbox{-a.e.}\,\,\mbox{on}\,\,  D, \alpha\ge 0.
\end{equation}
By \cite[Theorem 1]{mikulevicius} we infer that the  resolvent
$R^D_\alpha$ is strongly Feller for each ${\alpha\ge 0}$, i.e.
\begin{equation}
\label{eq.uewo.sf}
R^D_\alpha( B_b(D))\subset C_b(D),\quad \alpha\ge 0.
\end{equation}

{Using the kernel $r^D_\alpha(\cdot,\cdot)$ we can define the
  adjoint (with respect to the Lebesgue measure)
  resolvent operator  $\hat R^D_\al :L^{d/(d-1)}(D)\to L^{d/(d-1)}(D)$ given by 
\begin{equation}
\label{020205-20d}
\hat R^D_\al \psi(y)=\int_{D}r^D_\al(x,y)\psi(x)\,dx,\quad y\in
D,\,\psi \in L^{d/(d-1)}(D).
\end{equation}
We say that it is strongly Feller, if  the analogue of
\eqref{eq.uewo.sf} holds for $\hat R^D_\al$.}

\subsection{The   proof of Theorem \ref{thm020809-20}}

We have already concluded that $R^D_\al: L^d(D)\to L^d(D)$ is
compact for any $\al\ge 0$. In addition $\rho(R_1)$ - the  spectral
radius of $R^D_1$ - belongs to $[0,1]$. Moreover, by
\eqref{eq.uewo.st}, operator $R^D_1$
is {\em irreducible}, i.e.
\begin{equation}
\label{070909-20}
\mbox{ for any $f\in B^+(D)$ such that
  $\int_Df\,dx>0$ we have $R^D_1f(x)>0,\, x\in D$}.
\end{equation}
By virtue of the  Jentzsch Theorem, see
Theorem 6.6. of \cite{schaefer}, the spectral radius $\rho(R^D_1)$ is
positive and it equals to the (unique) simple eigenvalue
of $R^D_1$. The corresponding eigenvector  
 $\varphi_D$ is strictly positive on $D$. 
{ With the help of the resolvent identity \eqref{res-id}, applied
  to $\varphi_D$, with
$\al+1$ in place of $\al$ and $\beta=1$, one can easily check that 
$$
R^D_{\al+1}\varphi_D(x)=\frac{\varphi_D(x)}{\rho^{-1}(R^D_1)
  +\al},\quad  \,\al>0,\,x\in D.
$$
Thus, inverting the Laplace transform, we obtain
$$
e^{-t}P^D_t\varphi_D(x)=\exp\left\{-t \rho^{-1}(R^D_1) \right\}\varphi_D(x),\quad t\ge0,
$$
for each $x\in D$. 
This in turn leads to \eqref{eigen}, with
$\la_D:=\rho^{-1}(R^D_1)-1$}.  {Clearly, $\lambda_D>0$ as otherwise, this would contradict \eqref{012904-20}.}
\qed

\subsection{Irreducibility measure and its application}

\label{sec10a}

We recall the notion of   an irreducibility
measure associated with  a non-negative kernel.
Given $\al>0$ we let $K_\al:D\times{\cal B}(D)\to[0,+\infty)$ be a
substochastic kernel defined by 
{$$
K_\al(x,B):= R^D_\alpha1_B(x)\quad\text{ for
}\quad (x,B)\in D\times{\cal B}(D).
$$ }

\begin{definition}{(cf \cite[Section 2.2]{Num})}
\label{df011607-20}
Suppose that $\nu$ is  {a  non-trivial}  Borel measure on $D$.
We say that
the kernel $K_\al$ is  $\nu$-irreducible, if
for any $B\in {\cal B}(D)$ such that $\nu(B)>0$ and $x\in D$ we have
{$$
K^n_\al(x,B):=\int_{D} K_\al(x,dy_1)\Big(\int_{D} K_\al(y_1,dy_2)\ldots \int_{D} K_\al(y_{n-1},B)\Big)>0
$$
 for some $n\ge1$. By convention $y_0:=x$.} Measure  $\nu$ is  called
then an {\em irreducibility measure} for $K_\al$.
 {The kernel is said to be irreducible if it is
  irreducible for some measure.}
{A measure $\mu$ is
called a {\em maximal irreducibility measure} for $K_\al$ if it is an
irreducibility measure  and  any irreducibility measure $\nu$ for the kernel
     satisfies $\nu\ll \mu$.}
\end{definition}

\begin{lemma}
\label{lm2.im9402}
   For each $\al\ge0$ there exists a pair  of   functions
$\bar\psi_{\al},\bar\phi_{\alpha}\in B^+_b(D)$ { such that
both $
\bar\psi_{\al}>0$ and $\bar\phi_{\alpha}>0$,  $m$-a.e., and
\begin{equation}
\label{RDabc}
R^D_\al f(x)\ge R^D_{\alpha+1}\bar \psi_{\al}(x) \int_D R^D_{\alpha+1} f \bar\phi_{\alpha} dx,\quad x\in
D,\quad f\in B^+_b(D).
\end{equation}}
\end{lemma}
\proof

  Estimate \eqref{070909-20} implies that the
kernel 
$
K_\al(\cdot,\cdot)$ is {\em $m$-irreducible}. Here, as we recall
$m$ is the $d$-dimensional Lebesgue measure. {By \eqref{020205-20} we
conclude also that $\int_DK_\al(x,\cdot)dx\ll m$.
 Using \cite[part ii) of Proposition 2.4]{Num} we infer that the
 Lebesgue measure is in fact a
 maximal irreducible measure for the kernel.}
By \cite[Theorem 2.1]{Num} 
there exist an integer $n\ge1$, a    function  $\psi_{\al}:D\rightarrow
[0,\infty)$ and a non-trivial   Borel measure  $\nu_{\al}$ on $D$,
 the so called {\em small function} and {\em small measure}, such that
 $\int_D \psi_{\al}dx>0$ and 
{\begin{equation}
\label{eq.sfsm1n}
\Big(R_{\alpha+1}^D\Big)^nf(x)\ge \psi_{\al}(x)\int_D fd \nu_{\al},\quad f\in B^+(D),\quad x\in D.
\end{equation}
We claim that in fact from \eqref{eq.sfsm1n} it follows that 
\begin{equation}
\label{eq.sfsm1}
R_{\alpha}^D  f(x)\ge R_{\alpha+1}^Df(x)\ge \psi_{\al}(x)\int_D fd \nu_{\al},\quad f\in B^+(D),\quad x\in D.
\end{equation}
Indeed, we obviously have $R_{\alpha+1}^Df\le R_{\alpha}^Df$ for  any $f\in B^+(D)$.
By the resolvent identity, see \eqref{res-id},  
\begin{equation}
\label{010310-20} R_{\alpha}^D
R_{\alpha+1}^D f=R_{\alpha+1}^D
R_{\alpha}^D f =R_{\alpha}^Df-
R_{\alpha+1}^D f\le R^D_{\alpha} f \quad \mbox{for any $f\in B^+(D)$}.
\end{equation} 
If $n\ge2$ in \eqref{eq.sfsm1n}, then we can write
\begin{equation}
\label{eq.sfsm1n1}
\begin{split}
&
R_{\alpha}^D \Big(R_{\alpha+1}^D\Big)^{n-2}f(x)\ge \Big(R_{\alpha}^D
R_{\alpha+1}^D\Big)\Big(R_{\alpha+1}^D\Big)^{n-2}f(x)\\
&
\ge \Big(R_{\alpha+1}^D\Big)^nf(x)\ge \psi_{\al}(x)\int_D fd \nu_{\al},\quad f\in B^+(D),\quad x\in D.
\end{split}
\end{equation}
Iterating this procedure, we conclude \eqref{eq.sfsm1}.}

Applying \eqref{eq.sfsm1}
for $R_{\alpha+1}^D f$ in place of $f$, we get
\begin{equation}
\label{eq.sfsm2}
R_{\alpha} ^Df(x)\ge   \psi_{\al}(x)
\int_Df d\bar\nu_{\al},\quad f\in B^+(D),\quad x\in D,
\end{equation}
 where 
$$
\bar\nu_{\al}(B):= \int_BR^D_{\alpha+1}1_Bd \nu_{\al},\quad 
 B\in{\cal B}(D).
$$ Applying $R_{\alpha+1}^D$ to both sides of
 \eqref{eq.sfsm2} and using   again \eqref{010310-20},
 we get 
 \begin{equation}
\label{eq.sfsm3}
R_{\alpha} ^Df(x)\ge   \bar\psi_{\al}(x)
\int_Df d\bar \nu_{\al},\quad f\in B^+(D),\quad x\in D,
\end{equation}
with 
$\bar \psi_{\al}:= R^D_{\alpha+1} \psi_{\al}>0$, $m$ a.e.
Applying   $R^D_{\alpha+1}$ to both sides of \eqref{eq.sfsm3}
and using \eqref{010310-20} we get 
\begin{equation}
\label{eq.sfsm31}
R_{\alpha} ^Df(x)\ge   R_{\alpha+1} ^D\bar\psi_{\al}(x)
\int_Df d\bar \nu_{\al},\quad f\in B^+(D),\quad x\in D.
\end{equation}
Invoking \eqref{070909-20} we infer  that $\bar \nu_{\al}\ll m$. 
Thus, there exists a non-trivial $\bar \phi_\alpha\in B_b^+(D)$ such that $d\bar \nu_{\al}=\bar \phi_\alpha\,dx$.
{By \eqref{eq.uewo.st}, $\bar \phi_\alpha$ is strictly positive, $m$-a.e.}.
{Therefore, applying \eqref{eq.sfsm31}
with $f$ replaced by  $R_{\alpha+1}^D f$  we conclude  \eqref{RDabc}.}\qed

\subsection{Continuity of the resolvent  to the boundary}

\label{sec7.5}

{Recall that a point $\hat x\in\partial D$ is called {\em Dirichlet regular} if
$
P_{\hat x}(\tau_D>0)=0.$
If all the points of the boundary of $D$ are Dirichlet regular, we say that $D$ is Dirichlet regular.
It is well known (see \cite[Theorem 1.12]{De}) that if
$(R_\alpha^D)_{\alpha> 0}$ is strongly Feller, i.e. \eqref{eq.uewo.sf} holds, and $\hat x \in\partial D$ is Dirichlet regular, then
\begin{equation}
\label{eq4.grd1}
\lim_{D\ni x\rightarrow \hat x} \E_{x}\tau_D=0.
\end{equation}
Therefore, for a Dirichlet regular $D$ and $g\in B^+_b(D)$ function $R^D_\al g$ is continuous in $\bar D$, and 
\[
\lim_{D\ni x\rightarrow \hat x} R^D_\al g(x)=0,\quad \mbox{for any $\hat x\in\partial D$.}
\]
}

\begin{proposition}
\label{prop021205-20}
Suppose that   $D$ is bounded and 
conditions A1), A2) are in force. If $D$ satisfies  the exterior ball
condition (see Definition \ref{def3.1z}), then $D$ is Dirichlet
regular.
\end{proposition}
\begin{proof}
 By \cite[Proposition VII.3.3]{BH}, to prove that $\hat x\in\partial
 D$ is Dirichlet regular  it is enough to  find a  function
 $\psi$ - the so called {\em barrier} - that is strictly positive, and   superharmonic on $D$, i.e. $-A\psi\ge 0$ on $D$,
such that 
\[
\mathop{\lim_{x\to \hat x}}_{x\in D}\psi(x)=0.
\]
 For $\hat x\in\partial D$ let $B(y_0,r)$ be the
exterior ball to $D$ at this point. {Suppose that  $M$ is so large that $D\subset B(y_0,M)$.  Define a compactly supported,
smooth function $\psi:\bbR^d\to\bbR$, such that 
$$
\psi(x;\hat x)=c\left(\frac{1}{r^\si}-\frac{1}{|x-y_0|^\si}\right)
\quad \mbox{for }\frac{r}{10}<|x-y_0|<M
$$
for some  $c,\si>0$.}
It is  elementary to check that $\psi$ fulfills the requirements for a
barrier, provided $c,\sigma>0$ are sufficiently large.
\end{proof}


\section{Some auxiliary results}

\label{sec8}

\subsection{An auxiliary function and its lower bound}
Given $\bar y\in\R^d$ and {$\beta>0$} define 
{\begin{equation}
\label{eta-al}
 \eta(x; \beta,\bar y,r)=e^{-\beta|x-\bar y|^2}-e^{-\beta r^2},\quad x\in\BR^d.
\end{equation}}
We shall use this function in the proofs of the strong maximum
principle and both the Hopf, and
quantitative Hopf lemmas expressed in Theorems \ref{strong-max} --
\ref{thm:main1}. For this purpose we shall prove various estimates of 
{$\eta(x; \beta,\bar y,r)$} that are formulated in the present section as
well as in Section \ref{sec8.1.4}.

Define   open annuli
\begin{equation}
\label{o-a}
V_*(\bar y;r)=B(\bar y,r)\setminus \bar B(\bar y,r/2),\qquad V^*(\bar y;r)=B(\bar y,3r/2)\setminus \bar B(\bar y,r/2).
\end{equation}
{For given $y\in\BR^d$ we let $y^2:= y\otimes y=[y_iy_j]_{i,j=1}^{d}$}.
We start with the 
following result.
\begin{lemma}
\label{lm010204-20a} 
Suppose that $D$ {is bounded}
satisfies the interior ball condition with $\frak r(\cdot)$ as in
Definition \ref{def2.3}. 
Then, for any   $K>0$ there exist  constants $r_0,C>0$ depending only
on ${\ga_{\bar {\cal O}_D}}, \|N\|_\infty,\|Q\|_\infty, \|b\|_\infty, {\|c\|_\infty}$  such that for every $\hat x\in\partial D$,  any interior ball 
$B(\bar y,r)$ at $\hat x$ (cf Definition \ref{def2.3}) with $r\in
(0,r_0\wedge \frak r(\hat x)]$, we have
\begin{equation}
\label{eta-1}
(A-c)\eta\left(x;Cr^{-2},\bar y,r\right)\ge K,\quad x\in  V^*(\bar y;r).
\end{equation}
\end{lemma}
\proof
Let $0<r\le 1$,  $B(\bar y,r)$ be an interior ball in $D$ at $\hat x$,
and   $
V^*:=V^*(\bar y;r)
$ (cf \eqref{o-a}), see
 Fig. \ref{fig1}. 
Let 
$
 \eta(x):=\eta(x;\beta,\bar y,r)$, $x\in\R^d$ be given by
 \eqref{eta-al}. 
Then,
\begin{align}
\label{011004-20}
&\nabla\eta(x)=-2\beta (x-\bar y)e^{-\beta|x-\bar y|^2},\\
&\nabla^2\eta(x)=e^{-\beta|x-\bar y|^2} \left(4\beta^2 (x-\bar y)^2-2\beta I_d\right).\notag
\end{align}
With the notation ${\rm Tr}({\bf Q}(x))$ for the trace of ${\bf
  Q}(x)$, we can write
\[
\sum_{i,j=1}^d\frac12
q_{i,j}(x)\eta_{x_ix_j}(x)=\beta e^{-\beta|x-\bar y|^2}
\left[\vphantom{\int_0^1}2 \beta (x-\bar y)^T\cdot {\bf Q}(x) (x-\bar y)-{\rm Tr}({\bf Q}(x))\right]
\]
and
\[
\sum_{i=1}^d b_i(x)\eta_{x_i}=-2 \beta e^{-\beta|x-\bar y|^2}b^T(x)\cdot (x-\bar y).
\]

\bigskip

\begin{figure}[ht]

\begin{center}
\begin{tikzpicture}[scale=0.8]
 


\node [above]  at (-3.8,0) {$D^\complement$} ;
\node [above]  at (4.2,-1) {$\partial D$} ;

\node [below]  at (-4,-1.5){$D$} ;
 \draw [dotted] (0,-2) circle (3cm);

\fill [black!5,even odd rule] (0,-2) circle[radius=3cm]
circle[radius=2cm];

\fill [black!5,even odd rule]   (0,-2) circle[radius=1cm] circle[radius=0cm];

 \draw [dotted] (0,-2) circle (1cm);

\draw [thick]  (0,0) -- (0,-2);

\draw [decorate,decoration={brace,amplitude=6pt},xshift=-4pt,yshift=0pt]
(-0.0,-2) -- (-0.0,0.0) node [black,midway,xshift=-0.5cm] 
{\footnotesize $r$};

\draw [densely dotted]  (0,-2) -- (-3,-2);

\draw [decorate,decoration={brace,amplitude=10pt},xshift=-4pt,yshift=0pt]
 (0.1,-2) -- (-2.9,-2) node [below,black,midway,yshift=-0.2cm]   {\footnotesize $3r/2$};


\draw [thick]  (0,-2)--(0.72,-2.72);
\draw [decorate,decoration={brace,amplitude=4pt}]
 (0,-2)--(0.72,-2.72) node [above,black,midway, xshift=0.2cm ,yshift=0.1cm]   {\footnotesize $r/2$};


\node [below]  at (-0.25,-3.25){\mbox{\small{ $V^*(\bar y,r)$}}} ;
 \draw [densely dotted] (0,-2) circle (2cm);

  \draw [thick] (-6,-2) [out=45,in=180] to (0,0) [out=0,in=150] to
  (4,-1) [out=330] to (8,0); 

\node [fill, draw, circle, minimum width=3pt, inner sep=0pt,
pin={[fill=white, outer sep=2pt]135:$\hat x$}] at (0,0) {};


\node [fill, draw, circle, minimum width=3pt, inner sep=0pt,
pin={[fill=none, outer sep=-7pt]245:  {$\bar y$}}] at (0,-2) {};


  \end{tikzpicture}
\caption{} 
\label{fig1}
\end{center}
\end{figure}

Using the uniform ellipticity assumption \eqref{la-K}, we get 
\begin{align}
\label{eq2.2.3}
\nonumber L\eta(x)&\ge   \beta e^{-\beta|x-\bar y|^2} \left\{2 \ga_{\bar {\cal O}_D} \beta|x-\bar y|^2
                   -  ({\rm Tr}({\bf Q}(x))-2b^T(x)\cdot
                    (x-\bar y))\right\}\\
&
\ge   \beta  e^{-9 \beta r^2/4} \left\{ \frac{\ga_{\bar {\cal O}_D}\beta r^2}{2} -
                                  \Big(\|{\rm Tr}\,{\bf
  Q}\|_\infty+\frac{3r}{2}\|b\|_\infty \Big)\right\},\quad x\in V^*.
\end{align}
We shall choose $\beta>0$, $r\in(0,1]$ in such a way that
\begin{equation}
\label{021004-20}
\beta r^2=\gamma_*:=\frac{4}{\ga_{\bar {\cal O}_D}}\Big(\|{\rm Tr}\,{\bf
  Q}\|_\infty+\frac{3}{2}\|b\|_\infty \Big).
\end{equation}
Then, remembering that $r\in(0,1]$, we get
\begin{equation}
\label{eq2.2.3b}
L\eta(x)\ge    {\frac{\beta\gamma_*\ga_{\bar {\cal O}_D}}{4 } }e^{-9\gamma_*/4},\quad x\in V^*.
\end{equation}
{Using  \eqref{S}, 
we can write
\begin{align}
\label{eq2.2.4-0}
S\eta=S_1\eta(x)+\int_{\BR^d}\frac{ \delta\eta(x;y)}{1+|y|^2}N(x,dy),
\end{align}
where
\begin{align}
\label{031004-20-a}
&
S_1\eta(x):=\int_{\BR^d}\frac{|y|^2\left(\eta(x+y)-\eta(x)\right)}{1+|y|^2}N(x,dy),\notag\\
&
\delta\eta(x;y):=\eta(x+y)-\eta(x)-\sum_{i=1}^dy_i\partial_{x_i}\eta(x).
\end{align}
Using the fact that
$$
\delta\eta(x;y)=\int_0^1d\theta_1\int_0^{\theta_1}\frac{d^2}{d\theta^2_2}\delta\eta(x;\theta_2y)d\theta_2=\int_0^1d\theta_1\int_0^{\theta_1}y\cdot (\nabla^2\eta)(x;\theta_2y)y\,
d\theta_2 
$$
and subsequently formula \eqref{011004-20} for $\nabla^2\eta$ we 
conclude that
for any $M>10$ 
\begin{align}
\label{eq2.2.4}
{S\eta=\sum_{j=1}^4S_j\eta},
\end{align}
where, cf \eqref{011004-20}, 
\begin{align}
\label{031004-20}
&
S_2\eta(x):=-2\beta\int_0^1d\theta_1
  \int_0^{\theta_1}d\theta_2\int_{|y|\le Mr}
e^{-\beta|x+\theta_2 y-\bar y|^2} \frac{
  |y|^2} {1+|y|^2}N(x,dy),\notag\\
&
S_3\eta(x):=-2\beta\int_0^1d\theta_1
  \int_0^{\theta_1}d\theta_2\int_{|y|> Mr}
e^{-\beta|x+\theta_2 y-\bar y|^2} \frac{
  |y|^2} {1+|y|^2}N(x,dy),\\
&
S_4\eta(x):=4\beta^2\int_0^1d\theta_1 \int_0^{\theta_1}d\theta_2\int_{\BR^d}
e^{-\beta|x+\theta_2 y-\bar y|^2} {\frac{\big[(x+\theta_2 y-\bar y)\cdot
  y\big]^2} {1+|y|^2}}N(x,dy),\notag
\quad
  x\in V^*.
\end{align}}
We have $S_4\eta(x)\ge0$, therefore
\begin{align}
\label{eq2.2.4a}
S\eta(x)\ge \sum_{j=1}^3S_j\eta(x), \quad
  x\in V^*.
\end{align}
Furthermore, since $\|\eta\|_\infty=1$ we get, see \eqref{Nx1},
\begin{equation}
\label{051004-20}
|S_{1}\eta(x)|\le {2N_*},\quad x\in V^*.
\end{equation}
In addition, 
\begin{equation}
\label{061004-20}
|S_{2}\eta(x)|\le \beta n(Mr),\quad x\in V^*,
\end{equation}
where 
$$
 n(r):=\sup_{x\in \bar V^*}N_{B(0,r)}(x).
$$
Using assumption A2) and Dini's uniform convergence theorem we
conclude that 
\begin{equation}
\label{Dini}
\lim_{r\to0+}n(r)=0.
\end{equation}
Finally, we can write
\begin{equation}
\label{071004-20}
|S_3\eta(x)|\le 2\beta{\Theta(x,Mr)},
\end{equation}
where
$$
\Theta(x,r):=\int_0^1d\theta_2
  \int_{\theta_2}^1d\theta_1\int_{|y|> r}
e^{-\beta|x+\theta_2 y-\bar y|^2} \frac{
  |y|^2} {1+|y|^2}N(x,dy).
$$
We have
$$
|x+\theta_2 y-\bar y|\ge \left(\theta_2M-\frac{3}{2}\right)r\ge \frac{Mr}{2},
$$
provided that 
$$
\frac{r}{2}+\frac{3}{2M}=:\theta_*\le \theta_2\le 1.
$$
We can estimate 
\begin{align*}
&
\Theta(x,Mr)\le \int_0^{\theta_*}d\theta_2
  \int_{\theta_2}^1d\theta_1\int_{|y|> Mr}\frac{
  |y|^2} {1+|y|^2}N(x,dy)\\
&
+\int_{\theta_*}^1d\theta_2
  \int_{\theta_2}^1d\theta_1 e^{-\beta(M r)^2/4}\int_{|y|> Mr}
\frac{
  |y|^2} {1+|y|^2}N(x,dy)\\
&
\le \left(\theta_*+ e^{-\beta(M
  r)^2/4}\right) N_*\le \left(\frac{r}{2}+\frac{3}{2M}+ e^{-M^2\ga_*/4}\right) N_*\\
&
= \left(\frac{\sqrt{\ga_*}}{2\sqrt{\beta}}+\frac{3}{2M}+ e^{-M^2
 \ga_*/4}\right) N_*.
\end{align*}
Using estimates \eqref{eq2.2.3b} and \eqref{051004-20} -- \eqref{071004-20}  we get
\begin{align}
\label{A-est}
&
(A-c)\eta(x)\ge{\frac{\beta\gamma_*\ga_{\bar {\cal O}_D}}{4 }} e^{-9\gamma_*/4}-2
N_*-\beta n\left(\frac{M\sqrt{\ga_*}}{\sqrt{\beta}}\right)\\
&
 -
2\left(\frac{\sqrt{\beta\ga_*}}{2}+\frac{3\beta}{2M}+ \beta e^{-M^2
 \ga_*/4}\right) N_*-\|c\|_\infty\notag\\
&
=\beta\left\{{\frac{ \gamma_*\ga_{\bar {\cal O}_D}}{4 }} e^{-9\gamma_*/4}- n\left(\frac{M\sqrt{\ga_*}}{\sqrt{\beta}}\right)-2
\left(\frac{\sqrt{\ga_*}}{2\sqrt{\beta}}+\frac{3}{2M}+ e^{-M^2
 \ga_*/4}\right) N_*\right\}\notag\\
&
-2
N_*-\|c\|_\infty.\notag
\end{align}
Therefore, for  a $\beta_0\gg M^2\gg1$, we get that the right-hand side of the above inequality
is greater than $K$ on $V^*$.  By (\ref{021004-20}) we conclude
\eqref{eta-1} with {$r_0=\sqrt{\ga_*/ \beta_0}$}.\qed



\subsection{A weak subsolution of \eqref{eq3.1}}


\begin{definition}[{Weak subsolution, supersolution and solution of \eqref{eq3.1}}]
\label{df3.1g}
Let ${c,g\in B_b(\BR^d)}$.  A function $u\in  B_b(\BR^d)$ is called a {\em weak subsolution}
to (\ref{eq3.1}) in   $D$
if for any $x\in D$,
\begin{equation}
\label{subsol.g}
u(x)\le \mathbb E_x[e_c(t\wedge\tau_D)u(X_{t\wedge\tau_D})]+\mathbb E_x\left[ \int_0^{t\wedge\tau_D}e_c(s) g(X_s)\,ds\right],\quad t\ge 0.
\end{equation}
{The notions of a supersolution and solution can be then introduced
analogously as in Definition \ref{df3.1}.}
\end{definition}

\begin{remark}
By  Proposition \ref{prop010204-20} if  $u\in C_b(\BR^d)\cap
W^{2,d}_{\rm loc}(D)$  is a subsolution to (\ref{eq3.1}) on $D$, then
$u$ is a  weak subsolution to (\ref{eq3.1}) on $D$.
\end{remark}

{The Feynman-Kac representation of
 solutions, see Proposition \ref{prop010204-20}.
    holds also for weak  subsolutions by their very
definition. In particular, we have the following   form of a weak
maximum principle, that  follows by the same argument as 
Proposition \ref{prop012804-20}, see Section \ref{sec5.1b}.}
\begin{proposition}
\label{prop012904-20}
If $u$  is a weak subsolution of
\eqref{eq3.1a} {on a bounded and open set $D$}, then \eqref{032804-20} is in force.
\end{proposition}

\subsection{Properties  of $w_{c,D}$}
\label{sub.sub12}

Recall that $w_{c,D}(\cdot)$ is given by \eqref{wcD}. 
Obviously $0\le w_{c,D}(x)\le 1,\, x\in D$.
{Define $\hat w_{c,D}:=\mathbf1_{D}w_{c,D}$}, and  a martingale
$$
M_t:=\mathbb E_x\Big[e_c(\tau_D)|\FF_t\Big],\quad t\ge0.
$$where $e_c(t)$ is defined in \eqref{ecD}. 
Suppose that $x\in D$. {By the Markov property we have
$$
M_t=e_c(\tau_D)1_{[\tau_D\le t]}+v_{c,D}(X_t) e_c(t)1_{[\tau_D> t]},
$$
cf \eqref{eq2.pvf}.} Hence,
\begin{equation*}
 1_{[\tau_D> t\wedge
  \tau_D]} w_{c,D}(X_{t\wedge \tau_D})=1-e_c^{-1}(t\wedge \tau_D) M_{t\wedge
  \tau_D},\quad t\ge0. 
\end{equation*}
Since $\hat w_{c,D}(X_{t\wedge \tau_D})=1_{[\tau_D> t\wedge
  \tau_D]}w_{c,D}(X_{t\wedge \tau_D})$, we have
\begin{equation}
\label{wcD1}
{\hat w_{c,D}(X_{t\wedge \tau_D})}=1-e_c^{-1}(t\wedge \tau_D) M_{t\wedge
  \tau_D},\quad t\ge0. 
\end{equation}

Applying It\^o's formula to $e_c^{-1}(t) M_t $ yields
\begin{equation}
\label{011007-20}
e_c^{-1}(t) M_t =M_0+\int_0^t e_c^{-1}(s)\, dM_s+\int_0^t e_c^{-1}(s) M_s c(X_s)\,ds.
\end{equation}
Let $N_t$ be the second term in the right-hand side of the above
equality. It is a local martingale. 
{Thus, substituting from \eqref{011007-20} into \eqref{wcD1} and then
replacing $t$ by 
$t\wedge\tau_V$, where  $V\subset D$ is open, we get
\begin{align}
\label{eq.hatw123}
\nonumber
 {\hat w_{c,D}(X_{t\wedge\tau_V})}&=1-M_0-N_{t\wedge\tau_V}-\int_0^{t\wedge\tau_V} M_s e_c^{-1}(s) c(X_s)\,ds\\&
=w_{c,D}(X_0)-N_{t\wedge\tau_V}-\int_0^{t\wedge\tau_V} \big(cv_{c,D}\big)(X_s)
                                                                                           \,ds,\quad t\ge 0.
\end{align}
Therefore, $\hat w_{c,D}$ is a weak solution to $-Af=cv_{c,D}$ in
  $V$. Applying the It\^o formula to $\hat w_{c,D}(X_{t\wedge
    \tau_V})e_c(t\wedge \tau_V)$  we conclude also that $\hat w_{c,D}$
  is also a weak
  solution to
\begin{equation}
\label{eq.wcd0}
-Af+c f=c,\quad \mbox{in any open }\,\, V\subset D.
\end{equation}}


{From \eqref{eq.hatw123}}, we also get that 
\begin{equation}
\label{eq.wcd1}
{ w_{c,D}(x)}= R^D(c v_{c,D})(x)=R^D(c-c w_{c,D} )(x),\quad x\in D.
\end{equation}
By the very definition of $w_{c,D}$ we have
\[
w_{c,D}(x)>0\,\quad\mbox{iff}\quad \,{P_x\Big(\int_0^{\tau_D}c(X_r)\,dr>0\Big)>0},
\]
which in turn is equivalent to $R^Dc(x)>0$.
The latter however is a direct consequence of
\eqref{eq.uewo.st}. 
Thanks to \eqref{eq.uewo.sf} and \eqref{eq.wcd1} we conclude also
that $w_{c,D}$ is continuous in $D$. 
Summarizing we have shown the following.
\begin{lemma}
\label{lm010809-20}
The function $w_{c,D}$ is strictly positive and continuous on $D$.
\end{lemma}

\bigskip

The following result shall be used in the proof of estimate
\eqref{eta-wz} formulated below.
\begin{lemma}
\label{lm010704-20y}
Suppose that  A3'') holds. Then for any open set
 $V$, compactly embedded in $D$, there exists $a_*>0$ depending only on
$\|N\|_\infty,\|Q\|_\infty, \|b\|_\infty$,
 $\underbar{c}>0$  and $r_*:={\rm dist}(V,D^c)>0$ such that
 \begin{equation}
\label{040704-20}
 \inf _{x\in\bar V} w_{c,D}(x)\ge a_*.
 \end{equation}
 \end{lemma}
 \begin{proof}
 Let $\bar x\in D$ and $r>0$ be such that $B(\bar x,r)\subset
 D$. Observe that for any $t>0$,
 \[
 w_{c,D}(\bar x)\ge 1-\E_{\bar x}e^{- \underline c \tau_D}\ge 1-\E_{\bar x}e^{-\underline c \tau_{B(\bar x,r)}}\ge (1-e^{-\underline c t})P_{\bar x}(\tau_{B(\bar x,r)}>t).
 \]
 Let $\rho_{\bar x,r}\in C^2_b(\BR^d)$ be an arbitrary function taking
 values in $[0,1]$ and satisfying
 \[
 \rho_{\bar x,r}(\bar x)=1,\quad 1- \rho_{\bar x,r}(x)\ge a>0,\,\,x\in B^c(\bar x,r).
 \] 
 Since the process
 \[
 M_t:= 1- \rho_{\bar x,r}(X_{t\wedge \tau_{B(\bar x,r)}})+\int_0^{t\wedge \tau_{B(\bar x,r)}} A \rho_{\bar x,r}(X_s)\,ds,\quad t\ge 0
 \] 
 is a zero mean martingale we can write
 \begin{align}
\label{extra}
 P_{\bar x}(\tau_{B(\bar x,r)}\le t)&\le \frac1a\E_{\bar x}\left[(1- \rho_{\bar x,r})(X_{t\wedge \tau_{B(\bar x,r)}})\right]=-\frac1a \E_{\bar x}\left[\int_0^{t\wedge \tau_{B(\bar x,r)}} A \rho_{\bar x,r}(X_s)\,ds\right]\notag\\&
 \le \frac 1a E_{\bar x}(t\wedge \tau_{B(\bar x,r)}) \sup_{B(\bar x,r)} (A \rho_{\bar x,r})^- \le \frac ta\sup_{B(\bar x,r)} (A \rho_{\bar x,r})^-
\\& \le \frac ta C_A(\|\rho_{\bar x,r}\|_\infty+\|D^2 \rho_{\bar x,r}\|_\infty) \notag
 \end{align}
for some constant $C_A$ depending on the coefficients of the operator
$A$.

 Taking $r>0$ sufficiently small, so that $V\Subset D_r$ and letting $\rho_{\bar
  x,r}(x)=e^{-|x-\bar x|^2/r^2}$ for any $\bar x\in V$, we can choose
$t$ sufficiently small so that 
\eqref{040704-20} is in force.
 \end{proof}

\bigskip

\subsection{Some upper bounds on $\eta(x;\beta,\bar y,r)$}

\label{sec8.1.4}
\bigskip

In the present section we prove  some upper bounds on $\eta(x;\beta,\bar
y,r)$ formulated with the help of  function $w_{c,D}(x)$. They are
instrumental    in proving the quantitative Hopf lemmas expressed
in  Theorems  \ref{thm:main1a} and \ref{thm:main1}.

Recall that $D_{r}=[x\in D: {\rm dist}(x,\partial D)>r]$ for $r>0$ ,
cf \eqref{Da}. Let
$
r_*:=\inf[r>0:\,D_{r}=\emptyset].
$  
Define 
\begin{equation}
\label{eq.uewo.dism}
\rho_{c,A,D}(r):= \inf_{x\in D_{r}} w_{c,D}(x),\quad r\in(0,r_*).
\end{equation}
 We adopt the convention that $\rho_{c,A,D} (r):=\lim_{\xi\to r_*^-}\rho_{c,A,D}(\xi)$, if
$r\ge r_*$. Clearly, $\rho_{c,A,D}$ is nondecreasing and takes values in
$(0,1]$. Since   $w_{c,D}$ is a strictly  positive and continuous
 on $D$, see Lemma \ref{lm010809-20}, we conclude that $\rho_{c,A,D}(\cdot)$ is also strictly positive
 and  
 continuous on $(0,+\infty)$.
We start with the following.
\begin{lemma}
\label{lm010204-20}
Suppose that assumption A3') is fulfilled
and $D$ {is bounded} and 
satisfies the interior ball condition.  Let $r_0,C$ be constants as in
Lemma \ref{lm010204-20a} stated for  $K=1$.
Then, there exists a nondecreasing, strictly positive continuous function
$\rho_{c,A,D}:(0,+\infty)\to(0,1]$ depending only on
$c$,  of operator $A$ and the domain $D$ such that for any  $\hat x\in\partial D$ and    interior ball 
$B(\bar y,r)$ at $\hat x$ (cf Definition \ref{def2.3}) with $r\in
(0,r_0\wedge \frak r(\hat x)]$, we have
\begin{equation}
\label{eta-w}
\eta(x;Cr^{-2},\bar y,r)\rho_{c,A,D}(r/2)\le w_{c,D}(x),\quad x\in \bar V_*(\bar y;r).
\end{equation}
\end{lemma}

\proof
By Lemma \ref{lm010204-20a}, for given $\hat x
\in \partial D$, and any   interior ball 
$B(\bar y,r)$ at $\hat x$ (cf Definition \ref{def2.3}) with $r\in (0,\frak r(\hat x)\wedge r_0]$, function
$\eta(\cdot):=\eta(\cdot;Cr^{-2},\bar y,r)$ satisfies \eqref{eta-1}.  
This in particular means that  $g:=-A\eta<0$ in $\bar V_*(\bar y;r)$.
For any $\varepsilon\in(0,1]$ we let $\eta_\varepsilon:=
\varepsilon\eta$. Clearly, $\eta_\varepsilon\le 1$ (as $\eta\le1$) and 
\[
-A\eta_\varepsilon+c \eta_\varepsilon=c\eta_\varepsilon +\varepsilon
g,\quad \mbox{in}\,\,  { V_*(\bar y;r)}.
\]
{By \eqref{eq.wcd0}, $\hat w_{c,D}$ is a weak solution to 
\[
-Au+cu=c,\quad \mbox{in}\,\,  {V_*(\bar y;r)}.
\]
}
Therefore, {$\eta_{\varepsilon}-\hat w_{c,D}$} is a weak subsolution of
\eqref{eq3.1a} on $ { V_*(\bar y;r)}$.
Since $\eta\le1$ we have
\begin{equation}
\label{wcdeta}
\eta_{\varepsilon}(x)\le  {\hat w_{c,D}(x)},\, \quad x\in\bar B(\bar y,r/2),
\end{equation}
with  (cf
\eqref{eq.uewo.dism})
\begin{equation}
\label{varep}
\varepsilon:= \inf_{D_{r/2}}w_{c,D}=\rho_{c,A,D}(r/2).
\end{equation}
We also have  $\eta_\varepsilon\le 0\le  {\hat w_{c,D}}$ on $\overline B^c(\bar
 y,r)$, therefore by Proposition
 \ref{prop012904-20} we have  
\begin{equation}
\label{012402-21}
\eta_\varepsilon(x) - {\hat w_{c,D}(x)}\le \sup_{y\in  V_*^c(\bar y;r)}[\eta_\varepsilon(y)-
 {\hat w_{c,D}(y)}]\le 0,\quad x\in  {V_*(\bar y;r)}
\end{equation}
and the conclusion of the lemma follows.
 \qed

\begin{lemma}
\label{lm010204-20z}
Suppose that assumption A3'') is fulfilled
and $D$ {is bounded} and 
satisfies the interior ball condition.  Then, for any  $\hat x\in\partial D$ and  any interior ball 
$B(\bar y,r)$ at $\hat x$ there exist $C>0$, depending only on
$\ga_{\bar {\cal O}_D}, \|N\|_\infty,\|Q\|_\infty, \|b\|_\infty$,
and $a>0$ depending on the above parameters, {radius $r$} and $\underbar{c}>0$ such that
\begin{equation}
\label{eta-wz}
a \eta(x; Cr^{-2},\bar y,r)\le w_{c,D}(x),\quad x\in \bar V_*(\bar y;r).
\end{equation}
\end{lemma}
\proof
Suppose that
$r>0$ is as in the proof of Lemma \ref{lm010204-20}. Thanks to Lemma
\ref{lm010704-20y} we can choose $a_*>0$, depending   on the same
parameters as in the statement of the lemma,  such that
$\varepsilon\ge a_*$, where $\varepsilon$ is defined in
\eqref{varep}. {This allows us to conclude \eqref{eta-wz} from
\eqref{012402-21}. Recall that
$\eta_\eps=\eps\eta$ in \eqref{wcdeta}.}  \qed

\section{Proofs of the main results}

\label{sec5}

\subsection{Proof of Theorem \ref{strong-max}}

\label{sec4.2}

\subsubsection*{Proof of part 1): the strong maximum principle}
We start with the argument for the strong maximum principle. Recall that $u\in
W^{2,p}_{\rm loc}(D)\cap C_b(\bbR^d)$. 
The proof  is carried out
 by a contradiction. {We may assume that $D$ is bounded, since
 otherwise we would consider an intersection of $D$ with an open
 ball of a sufficiently large radius, that includes a maximal point.} Suppose that there exist
$P,Q\in D$ such that $u(P)=M\ge0$ and $u(Q)<M$. In fact, see the proof of
Theorem 3.5, p. 61 of \cite{weinberger}, we can assume that there
exists a ball $B(\bar y,3r/2)\subset D$ such that $u(\hat x)=M$ for some $\hat
x\in \partial B(\bar y,r)$ and $u(x)<M$ for all $x\in B(\bar y,r)$. By
virtue of Lemma \ref{lm010204-20a} we can assume that  there exists
also a function $\eta(\cdot;\beta,\bar y,r)$  such that
\eqref{eta-1} holds on  
 $V^*:=V^*(\bar y;r)$, cf \eqref{o-a}.
We can choose $\eps>0$ sufficiently small so that
\begin{equation}
\label{v}
v(x):=u(x)+\eps \eta(x;\beta,\bar y,r)< M\quad\mbox{ for }x\in  \bar
B(\bar y,r/2).
\end{equation}
We also have $v(x)\le M$ on $\bar B^c(\bar y,r)$ (since then  $\eta(x;\beta,\bar y,r)<0$).
Since both $u(\cdot)$ and $\eta(\cdot)$ are
subsolutions to  \eqref{eq3.1a} on $V^* $ so is
$v(\cdot)$.
Let
$\tilde v(x):=v(x)-M$. It is a
subsolution to \eqref{eq3.1a} on $V^* $. Therefore by Proposition
\ref{prop012804-20} we have
$$
\sup_{x\in V^*} \tilde v(x)\le \sup_{x\in (V^*)^c} \tilde v(x)<0,
$$
which leads to a contradiction, as $0=\tilde v(\hat x)$ at $\hat x\in
V^* $.

\subsection*{Proof of part 2): the Hopf lemma}
Concerning the proof of \eqref{010104-20a} 
we essentially follow the classical proof of the Hopf lemma for
a subsolution
of a uniformly elliptic equation, as presented in
e.g. \cite{weinberger}, see Theorem 3.7, p. 65. Since $D$ satisfies
the interior ball condition, thanks to  Lemma \ref{lm010204-20a},
we can find an interior ball $B(\bar y;r)$ at $\hat x$ and $\beta>0$ such that
$$
(A-c)\eta(x;\beta,\bar y,r)>0,\quad 
x\in \bar V_*,
$$
where $ V_*:= V_*(\bar y;r)$.
Thanks to the already proved strong maximum principle for $u$ and our
assumption that $u\not\equiv M$ on $D$,
we know that $u(x)<M$ for $x\in B(\bar y;r/2)$.   {Since  $\eta\le 1$
  on $B(\bar
  y;r/2)\subset D_{r/2}$, see \eqref{Da},  we have 
\begin{equation}
\label{011504-20}
v(x):=u(x)+\eps \eta(x;\beta,\bar y,r)<M\quad\mbox{ for  }x\in  B(\bar
y;r/2).
\end{equation} with
\begin{equation}
\label{eq.qhler4}
\varepsilon:=  \inf_{y\in D_{r/2}}(M-u(y)).
\end{equation}
We also have
$v(x)<M$ for $x\in \bar B^c(\bar y,r)$ (as then $\eta\le0$).  
Therefore, by Proposition \ref{prop012804-20} applied to $v$ on  $\bar V_*$, we conclude that
$$
v(x)\le M=u(\hat x)=v(\hat x),\quad x\in \bar V_*,
$$
which yields $\partial_{{\bf n}} v(\hat x)\ge0$. Using
\eqref{011504-20} we get 
$$
\underline{\partial}_{{\bf n}} u(\hat x)\ge 2\beta\eps re^{-\beta r^2} >0.
$$
By the construction of $\eta(\cdot; \beta,\bar y,r)$, the parameters $\beta,r$ depend only on
${\gamma_{\bar{\mathcal O}_D}}, M_A$. This ends the proof of Theorem \ref{strong-max}.\qed

\begin{remark}  {
 We follow the notation of the proof of Theorem
 \ref{strong-max}.  From the proof of part 2) of the
 theorem we conclude that
\begin{equation}
\label{eq.qhlr857}
u(\hat x)-u(x)\ge   \eta(x) \inf_{y\in D_{r/2}}\big(M-u(y)\big),\quad x\in \overline{ V}_*,
\end{equation}
where, for simplicity, we let $\eta(x):=\eta(x; \beta,\bar y,r)$.

{ Suppose that $r_0>0$ satisfies
\begin{equation}
\label{br0}
2 \beta r_0^2<1.
\end{equation} Recalling the definition of $\eta(\cdot)$ (see
\eqref{eta-al}) we infer from \eqref{eq.qhlr857} that
for any $r\in (0,r_0]$:
\begin{equation*}
\begin{split}
&
\frac{u(\hat x)-u(\hat x-tn(\hat x,\bar y))}{t}\ge \frac{1}{t}\Big[\eta\big(\hat x-tn(\hat x,\bar y)\big)- \eta(\hat x)\Big]\inf_{y\in
  D_{r/2}}\big(M-u(y)\big)\\
&
\ge \underline{\partial}_{{\bf
    n}} \eta(x)_{|x=\hat x-tn(\hat x,\bar y)}\inf_{y\in
  D_{r/2}}\big(M-u(y)\big)\ge  \beta r
e^{-\beta r^2/4} \inf_{y\in D_{r/2}}\big(M-u(y)\big),\quad t\in {(0, r/2]},
\end{split}
\end{equation*}
The last inequality follows from the concavity of $t\mapsto\eta\big(\hat
x-tn(\hat x,\bar y)\big)$ for $t\in[0,r/2]$, due to  \eqref{br0} (see
\eqref{011004-20}). Summarizing, we have shown that
\begin{equation}
\label{eq.qhlr857i}
\frac{u(\hat x)-u(\hat x-tn(\hat x,\bar y))}{t}\ge  \beta r
e^{-\beta r^2/4} \inf_{y\in D_{r/2}}\big(M-u(y)\big),\quad t\in {(0, r/2]},
\end{equation}
where $n(\hat x,\bar y):=(\hat x-\bar y)/|\hat x-\bar y|$.}}
\end{remark}



{The argument leading to \eqref{eq.qhlr857} can be also used to obtain
a lower bound on the resolvent operator applied to a non-negative
function and we can formulate the following result.}
\begin{proposition}
\label{prop010310-20}
Assume that a  {bounded}   domain $D$ satisfies  the  uniform interior  ball condition. 
Suppose that  {non-trivial} $f\in B^+(D)$  is such that $R^D_\al f\in C_0(\bar
D)$ for some $\al\ge0$. Then, there exists $a>0$ such that 
\begin{equation}
\label{030310-20}
R^D_{\al}f(x)\ge a\delta_D(x),\quad x\in D.
\end{equation}
\end{proposition}
\proof We use the notation from the proof of Theorem \ref{strong-max}.
It can be easily seen, by an application of the It\^o formula, that $u=-R^D_{\al}f(x)$ is a weak subsolution of  $(A-\al)v=0$ on $D$.
By virtue of  \eqref{eq.uewo.st} we
have $u(x)<0$, $x\in D$. We can invoke an analogue of
\eqref{eq.qhlr857} that also holds for weak subsolutions and conclude
that
  there exists $a'>0$ such that for any $\hat x\in \partial D$
 \begin{equation}
 \label{eq.qhlr857a}
\begin{split}
 &{R^D_{\al}f (\hat x-tn(\hat x,\bar y))}={-u(\hat x-tn(\hat x,\bar
   y))}\ge  a' t\inf_{y\in D_{r/2}}\big(-u(y)\big)\\
&
=a' t\inf_{y\in
   D_{r/2}}R^D_{\al}f(y),\quad t\in (0, 1],
\end{split}
 \end{equation}
where $n(\hat x,\bar y):=(\hat x-\bar y)/|\hat x-\bar y|$. This, in
turn, implies \eqref{030310-20}.
\qed


\subsection{Proof of Theorem \ref{thm:main1}}

\label{sec6.3}



Suppose that $\hat x$ is as in the statement of Theorem
\ref{thm:main1} and $u$ is a subsolution of  \eqref{eq3.1a}. 
The conclusion of the theorem is obvious if $u(\hat x)<0$ and follows directly from
\eqref{010104-20a}, if $u(\hat x)=0$.
Assume therefore that $u(\hat x)>0$.
{Using the Feynman-Kac formula, see Proposition \ref{prop010204-20},
we infer that for any $x\in D$
\begin{equation}
\label{020204-20}
u(\hat x)-u(x)\ge 
u(\hat x)-\E_x\left[e_c(\tau_D)u(X_{\tau_D})\right]\ge u(\hat x)\left(1-v_{c,D}(x)\right)= u(\hat x)w_{c,D}(x).
\end{equation} 
Invoking the lower bound on $w_{c,D}(x)$, contained in   Lemma
\ref{lm010204-20}, we conclude \eqref{eta-w}  for any interior ball $B(\bar y,r)$ for $D$ at $\hat x$, with $r\in (0,r_0\wedge \frak r(\hat x)]$.}
Therefore, for ${\bf n}$ that is the outward unit normal to $\partial B(\bar y,r)$
at $\hat x$, we have
\begin{equation}
\label{020204-20a}
\underline{\partial}_{{\bf n}} u(\hat x)\ge -\rho_{c,A,D}(r/2)u(\hat x)\partial_{{\bf n}}
\exp\left\{-\frac{C}{r^2}|x-\bar y|^2\right\}\vphantom{1}_{| x=\hat x}=\frac{2C}{r} e^{-C} \rho_{c,A,D}(r/2) u(\hat x).
\end{equation}
The conclusion of Theorem \ref{thm:main1}  then follows.

\qed




\subsection{Proof of Theorem \ref{thm:main1a}}

\label{sec6.3a}

We start with \eqref{020204-20}.
Invoking  the lower bound on $w_{c,D}(x)$, provided in  Lemma \ref{lm010204-20z}, we can find $B(\bar y,r)$ -  an
interior ball for $D$ at $\hat x$ - and
$C,a>0$ as in the statement of the lemma such that  \eqref{eta-wz} holds.
Therefore, for ${\bf n}$ that is the outward unit normal to $\partial B(\bar y,r)$
at $\hat x$, we have
\begin{equation}
\label{020204-20ax}
\underline{\partial}_{{\bf n}} u(\hat x)\ge -au(\hat x)\partial_{{\bf n}}
\exp\left\{-\frac{C}{r^2}|x-\bar y|^2\right\}\vphantom{1}_{| x=\hat x}=2\frac{Ca}{r^2} re^{-C} u(\hat x).
\end{equation}
The conclusion of Theorem \ref{thm:main1a}  then follows.






\subsection{Proof of Theorem \ref{pee}}

\label{sec7.2.1}

Using the Feynman-Kac formula  (Proposition \ref{prop010204-20}) we can write
\begin{equation}
\label{052904-20b}
u(\hat x)-u(x)\ge \big(1-\E_xe_c(\tau_D\wedge t)\big)u(\hat x)+\mathbb
E_x\left[\int_0^{\tau_D\wedge t}e_c(s)(Au-cu)(X_s)\,ds\right],\quad t>0.
\end{equation}
Observe that 
\begin{equation}
\label{052904-20}
\Big(1-\E_xe_c(\tau_D\wedge t)\Big)u(\hat x)\ge (1-e^{-\underline c t})P_x(\tau_D>t) u(\hat x).
\end{equation}
From \eqref{eigen} we have
\begin{equation}
\label{062904-20}
e^{-\lambda_D t}\varphi_D(x)=\E_x\left[\varphi_D(X_t)\mathbf1_{\{t<\tau_D\}}\right]\le \|\varphi_D\|_\infty P_x(\tau_D>t).
\end{equation}
Substituting into \eqref{052904-20} the lower bound on
$P_x(\tau_D>t)$ obtained from \eqref{062904-20}
and maximizing over $t>0$ we conclude that
\begin{equation}
\label{072904-20}
u(\hat x)-u(x)\ge \frac{u(\hat x)
  \varphi_D(x)}{\|\varphi_D\|_\infty}\cdot \frac{\underline
  c /\la_D}{(1+\underline c/\la_D)^{1 +\la_D/\underline c}}\ge \frac{u(\hat x)
  \varphi_D(x)}{e\|\varphi_D\|_\infty}\cdot\frac{\underline
  c}{\la_D+\underline c}.
\end{equation}
In the last inequality we have used an elementary bound
$(1+s)^{1/s}\le e$ {valid for all
$s>0$}.

Furthermore, since $c(x)\ge0$, we have  $\bar c:=\|c\|_\infty$ (cf
\eqref{c-c}) and
\begin{align}
\label{052904-20c}
&
\mathbb
E_x\left[\int_0^{\tau_D\wedge t}e_c(r)(Au-cu)(X_r)\,dr\right]\ge {\rm essinf}_D(Au-cu)\mathbb
  E_x\left[\int_0^{\tau_D\wedge t}e^{-\bar c r}dr\right]\notag
\\
&
\ge
  \frac{1}{\bar c}(1-e^{-\bar c t}) \,{\rm essinf}_D(Au-cu) P_x(\tau_D>t).
\end{align}
Using again \eqref{062904-20}  to estimate $P_x(\tau_D>t)$ from below
and maximizing over $t>0$ we conclude that  
\begin{align}
\label{052904-20a}
&u(\hat x)-u(x)\ge \frac{
  \varphi_D(x)}{\|\varphi_D\|_\infty \bar c}\cdot\frac{\bar
  c /\la_D}{(1+\bar c/\la_D)^{1 +\la_D/\bar c}}{\rm essinf}_D(Au-cu)\\
&
\ge \frac{
  \varphi_D(x)}{e\|\varphi_D\|_\infty }\cdot\frac{1}{\la_D+\bar c}{\rm essinf}_D(Au-cu).\notag
\end{align}
Estimate \eqref{012804-20} follows easily from \eqref{072904-20} and \eqref{052904-20a}.
\qed

\subsection{Proof of Theorem \ref{pee1}}

\label{sec7.2.2}

We start with the estimate 
\eqref{052904-20b}. The first term in the right hand side is bounded
in the same way as in Section \ref{sec7.2.1}. To estimate the second
term we note that, {thanks to finiteness of the moment of the exit time (Proposition \ref{lm010104-20}),}
\begin{equation}
\label{060909-20}
\lim_{t\to+\infty}\mathbb
E_x\left[\int_0^{\tau_D\wedge t}e_c(s)(Au-cu)(X_s)\,ds\right]{\ge R^D_{\bar c}\Big((A-c)u\Big)}.
\end{equation}
 {By Lemma \ref{lm2.im9402} there exist {strictly positive, $m$-a.e.} $\bar \psi, \bar \phi\in B^+_b(D)$
such that 
\begin{equation}
\label{RDa}
R^D_{\bar c} f(x)\ge R^D_{\bar c+1}\bar \psi(x) \int_D R^D_{\bar c+1}f \bar \phi \,dx,\quad x\in
D,\quad f\in B^+_b(D).
\end{equation}
{Thanks to the irreducibility and strong Feller properties of the
resolvent,} see  \eqref{eq.uewo.st} and \eqref{eq.uewo.sf}, we conclude that
$\psi:=R^D_{\bar c+1}\bar \psi$ is strictly positive and continuous in
$D$. By \eqref{eq.uewo.st}, the measure  
$$\nu(A)= \int_D \bar\phi
R^D_{\bar c+1} \mathbf{1}_A dx, \quad A\in {\cal B}(D)
$$
is equivalent to $m$. Thus, $ {\chi}:=d\nu/dx$ is strictly positive in $D$ ($m$-a.e.).
Applying \eqref{RDa} we
obtain 
\begin{equation}
\label{020310-20}
 R^D_{\bar c}\Big((A-c)u\Big)(x)\ge \psi(x)\int_D (Au-cu) \chi dx,\quad x\in D
\end{equation}
and \eqref{012804-20-1} follows.}\qed

\subsection{Proof of Theorem \ref{thm:main3a}}

\label{sec-bony}
Let $\varepsilon_0>0$ be such that $B(\hat x,2\varepsilon_0)\subset
D$. Set $V_\varepsilon=B(\hat x,\varepsilon)$. By  {the Feynman-Kac
formula}, see Proposition \ref{prop010204-20}, for any $\varepsilon\in (0,\varepsilon_0]$,
 \begin{equation}
\label{010204-20dfg1ccv}
u(\hat x)-\E_{\hat x}\left[u(X_{\tau_{V_\varepsilon}})\right]=-\E_{\hat x}\left[\int_0^{\tau_{V_\varepsilon}}(Au)(X_s)\,ds\right].
\end{equation}
Using the Ikeda Watanabe formula, see \cite[Remark 2.46, page 65]{BSW}, we get
\begin{equation}
\label{010293-20dfg1v}
\begin{split}
&\mathbb E_{\hat
  x}\left[\sum_{0<s\le\tau_D}\mathbf{1}_{D}(X_{s-})\mathbf{1}_{\overline
  D^c}(X_s)\right]\\
&
=\mathbb E_{\hat x}\left[\int_{\BR^d}\int_0^{\tau_D}\mathbf{1}_{D}(X_{s-})\mathbf{1}_{\overline D^c}
(y)N(X_{s-},dy-X_{s-})\,ds\right]=0.
\end{split}
\end{equation}
The last equality follows from the fact that $\mathcal S(D)\subset \overline D$.
This implies that $\mathbb P_{\hat x}(X_{\tau_{V_\varepsilon}}\in \overline D)=1$.

Since $\hat x\in D$ is a  global  maximal point of $u$ in $\overline
D$ (by continuity of $u$), from \eqref{010204-20dfg1ccv}, we obtain
\[
 \E_{\hat x}\left[\int_0^{\tau_{V_\varepsilon}}(Au)(X_s)\,ds\right]\le 0,\quad \varepsilon\in (0,\varepsilon_0].
\]
Combining the above with  \eqref{020205-20} (here  $D:=V_\varepsilon$ and $\alpha=0$),
\[
 \left(\essinf_{V_\varepsilon}Au\right)\int_{V_\varepsilon}r_0^{V_\varepsilon}(\hat x,y)\,dy\le  \E_{\hat x}\left[\int_0^{\tau_{V_\varepsilon}}(Au)(X_s)\,ds\right]\le 0,\quad \varepsilon\in (0,\varepsilon_0].
\]
By  \eqref{eq.uewo.st} we have $\int_{V_\varepsilon}r_0^{V_\varepsilon}(\hat x,y)\,dy>0$. Therefore, $\essinf_{V_\varepsilon}Au\le 0,\, \varepsilon\in (0,\varepsilon_0]$,
which proves the assertion of the theorem, cf \eqref{essinf}.
\qed

\subsection{Proof of Theorem \ref{thm:main4a}}
\label{prf-thm5.4}

 {
{Using  Lemma \ref{lm2.im9402} we conclude that
there exists a pair of  functions
$\bar\psi,\bar\phi\in B^+_b(D)$ such that
$\bar\psi>0$, $\bar\phi>0$, $m$ a.e.}
and
\begin{equation}
\label{RDabcbarc}
R^D_{\bar c+1} f(x)\ge R^D_{\bar c+2}\bar \psi(x) \int_D R^D_{\bar c+2} f \bar\phi dx,\quad x\in
D,\quad f\in B^+_b(D).
\end{equation}
Applying this inequality, with $f$ replaced by $R^D_{\bar c}f$, and
using \eqref{010310-20} we get
\begin{equation}
\label{010510-20}
R^D_{\bar c} f(x)\ge R^D_{\bar c+2}\bar \psi(x) \int_D R^D_{\bar c} f  d\nu,\quad x\in
D,\quad f\in B^+_b(D),
\end{equation}
where 
\begin{equation}
\label{nuB}
\nu(B)=\int_D\bar\phi R^D_{\bar c+2}\mathbf{1}_B\,dx, \quad B\in
{\cal B}(D).
\end{equation}
{By \eqref{eq.uewo.st}, $\nu\sim m$. Therefore, we can write
$\chi:D\to(0,+\infty)$ such that
\begin{equation}
\label{nuB1}
\nu(B)=\int_B\chi(x)\,dx, \quad B\in
{\cal B}(D),\quad \mbox{where }\chi(x)=\hat R^D_{\bar c+2}\bar\phi 
\end{equation}
and $\hat R^D_{\bar c+2}$ is defined by \eqref{020205-20d}.
By \eqref{eq.uewo.st} and 
\eqref{eq.uewo.sf} we infer that  $R^D_{\bar c+2}\bar \psi$
is continuous and strictly positive on $D$.}

In the next step, we  show that $R^D_{\bar c} f(x)$ in \eqref{010510-20}
can be replaced by $w$ - any $\bar c$-excessive function with respect to
$(P^D_t)_{t\ge 0}$, i.e. a non-negative function satisfying
\begin{equation}
\label{030510-20}
P^D_tw(x)\le e^{\bar c t}w(x),\quad t\ge0,\,x\in D
\end{equation}
and
\begin{equation}
\label{030510-20a}
\lim_{t\to0+}e^{-\bar c t}P^D_tw(x)=w(x),\quad x\in D.
\end{equation}
More precisely, for any $w$, that is  $\bar c$-excessive,  we have
\begin{equation}
\label{010510-20aa}
w(x)\ge R^D_{\bar c+2}\bar \psi(x) \int_D w  d\nu,\quad x\in
D.
\end{equation}
{Indeed, from \cite[Proposition II.2.6]{bg} we conclude that, there exists a
sequence of non-negative Borel functions $(f_n)_{n\ge 1} $ such that
$\big(R^D_{\bar c}f_n(x)\big)_{n\ge1}$ is monotonne increasing and
\begin{equation}
\label{040510-20}
\lim_{n\to+\infty}R^D_{\bar c}f_n(x) =w(x), \quad x\in D.
\end{equation}
We can write inequality \eqref{010510-20} for each $R^D_{\bar
  c}f_n(x)$. Estimate \eqref{010510-20aa} then follows by an
application of the
monotonne convergence theorem and \eqref{040510-20}.}}

Coming back to the proof of \eqref{830104-20aabb}, it follows from \eqref{010510-20aa},
provided we prove that $u$ is  $\bar c$-excessive. {Indeed for any
$V\Subset D$, by
\eqref{010510-20}, we can write
\begin{equation}
\label{010510-20aa1}
u(x)\ge R^D_{\bar c}u(x)\ge C\int_V\chi udx
\end{equation}
with $C:= \inf_{x\in V}R^D_{\bar c+2}\bar \psi(x) >0$  and $\chi$ is  defined in \eqref{nuB1}.}

Since $u$ is a supersolution to \eqref{eq3.1a}, 
by Proposition \ref{prop010204-20}, we have
\[
u(x)\ge \mathbb E_x\Big[e_c(\tau_D\wedge t)u(X_{\tau_D\wedge t})\Big],\quad t\ge 0,\, x\in D.
\]
Thus, by non-negativity of $u$, 
\[
u(x)\ge  e^{-\bar c t}\mathbb E_x[ u(X_{t})1_{[t<\tau_D]}]=e^{-\bar c t}P^D_tu(x),\, t\ge 0
\]
and \eqref{030510-20} holds. Equality \eqref{030510-20a}  is a
consequence of 
 continuity of $u$, in $\bar D$. This ends the proof of Theorem
 \ref{thm:main4a}.

\qed

\medskip

\begin{remark}
\label{rmk9.3}
 Under the assumptions of Theorem
   \ref{cor.d21.20.1}  (resp. Theorem \ref{cor.d21.20.1a}),
   using estimate \eqref{010510-20aa1} and the conclusion of Proposition
   \ref{prop010310-20}, we can prove that in \eqref{830104-20aabb}, 
 $C:=a{\rm dist}(V, D^c)$ (resp. $\nu:=a\delta_D\,dx$, and $C:={\rm
   dist}(V, D^c)$) for some $a>0$, depending only on $D$.
\end{remark}

\medskip

{
\begin{remark}
\label{rmk9.4}
 It is clear that if $\hat R_\al$ satisfy the strong Feller 
 condition for some $\al>0$, then $\chi\in C(D)$ and is strictly positive. As a result
 we 
 conclude then the usual form of the weak Harnack inequality, i.e. that
  for any $V {\Subset} D$ we can find a constant $C>0$   such that
for any non-negative supersolution $u$   we have
\begin{equation}
\label{830104-20aabbb}
\inf_{x\in V} u(x)\ge C \int_Vu\,dx.
\end{equation}
According to \cite[Theorem 1]{mikulevicius}, a sufficient condition for the aforementioned strong Feller
property of the adjoint resolvent is  that the
coefficients of the operator $\hat A_0$, given by \eqref{eqi.6.30.09},
satisfy the assumptions A1), A2) and A4).
\end{remark}
}



\subsection{Proofs of Theorems  \ref{cor.d21.20.1} and  \ref{cor.d21.20.1a}}
\label{sec7.2.2a}

  {Invoking the argument from the proofs of  Theorem \ref{pee1}   and
    Lemma  \ref{lm2.im9402} we conclude that there exist {strictly positive $m$-a.e.} $\bar\psi_0,\bar\phi_0\in B^+_b(D)$
such that 
\begin{equation}
\label{012804-20-9d}
u(\hat x)-u(x)\ge \frac{\underline
  c\varphi_D(x) u(\hat
x)}{2e
  \|\varphi_D\|_\infty (\la_D+\underline c)}+R^D_{\bar c+1}\bar \psi_0(x)\int_DR^D_{\bar c+1}(Au-cu)\bar\phi_0\,dx,\quad x\in D.
\end{equation}
By \eqref{eq.uewo.st}, measure 
$$\nu(B)=\int_D\bar\phi_0R^D_{\bar
  c+1}\mathbf{1}_B\,dx, \quad B\in {\cal B}(D)
$$ is equivalent to
$m$. Thus, $\phi_0:=d\nu/dx$ is strictly positive and 
\begin{equation}
\label{012804-20-9g}
\int_DR^D_{\bar c+1}(Au-cu)\bar\phi_0\,dx=\int_D(Au-cu)\phi_0\,dx.
\end{equation}
Since $D$ satisfies the exterior ball condition and $R^D_{\bar c+1}$ is
strongly Feller (see \eqref{eq.uewo.sf}), we have $R^D_{\bar c+1}\bar
\psi_0 \in C_0(\overline D)$, see Section \ref{sec7.5}.
Using Proposition \ref{prop010310-20} we conclude that there exists
$a_1>0$ such that   $R^D_{\bar c+1}\bar\psi_0\ge a_1\delta_D$. This
finishes the proof of Theorem  \ref{cor.d21.20.1}.}

Under assumptions of Theorem  \ref{cor.d21.20.1a}, we may  write
\begin{equation}
\label{012804-20-9j}
\int_DR^D_{\al_*}(Au-cu)\bar\phi_0\,dx=\int_D(Au-cu)\hat R^{\hat c,D}_{\al_*}\bar\phi_0\,dx,
\end{equation}
where {$\al_*:=(\bar c+1)\vee (\|\hat c\|_\infty+1)$}.
Invoking again   Proposition \ref{prop010310-20}  we conclude that there exists $a_2>0$ such that $\hat
R^{\hat c,D}_{\al_*}\bar\phi_0\ge a_2\delta_D$ and the conclusion of the
theorem follows.\qed

\subsection{Proof of Theorem \ref{thm013009-20}}

\label{sec-res-dual}

With no loss of generality we may and shall assume that $\hat c\le0$, as
otherwise we could have considered $A-\al$ for a sufficiently large $\al>0$.

 {Suppose that $p>d$. By \cite[Theorem 1.2]{Taira3}, the operators $(A,D(A))$ and  $(\hat A,D(\hat A))$ generate $C_0$-semigroups of contractions on $C_0(\overline D)$, where
\begin{align*}
&D(A)=[u\in W^{2,p}(D)\cap C_0(\overline D): Au\in C_0(\overline D)],\\
&
D(\hat A)=[u\in W^{2,p}(D)\cap C_0(\overline D): \hat Au\in C_0(\overline D)].
\end{align*}
By the standard approximation we conclude  that \eqref{eq.deo.1} holds for $u\in D(A), v\in D(\hat A)$.
Let $(G_\alpha)_{\alpha>0}$, $(\hat G_\alpha)_{\alpha>0}$ denote the
semigroup resolvents generated by $A$, $\hat A$, respectively. Using
the martingale problem we conclude that $G_\alpha=R^D_\alpha$ and $ \hat
G_\alpha=\hat R^{\hat c,D}_\alpha$, where as we recall $R^D_\alpha$, $
 \hat R^{\hat c,D}_\alpha$ are the resolvent corresponding to the
 probability transition semigroups \eqref{PDt} and \eqref{hPDt}.
Thus, 
$$
R^D_\alpha(C_0(\overline D))\subset C_0(\overline D)\cap
W^{2,p}(D)\quad\mbox{ and }\quad \hat R^{\hat
  c,D}_\alpha(C_0(\overline D))\subset C_0(\overline D)\cap W^{2,p}(D)
$$
 for any $p>d$ and $\alpha\ge 0$.
Let $(R^{D}_\alpha)^*$ denote the adjoint operator to $R^D_\alpha$ in $L^2(D;m)$. 
Then for any $w\in C_0(\overline D)$, there exists $\eta\in D(A)$ such that $(-A+\alpha )\eta=w$. 
By using \eqref{eq.deo.1}, for any $u\in C_0(\overline D)$
\begin{align*}
\int_D  w\hat R_\alpha^{\hat c,D}u  \,dx&=\int_D (-A+\alpha )\eta\hat R_\alpha^{\hat c,D}u\,dx\\&=
\int_D \eta (-\hat A+\alpha )\hat R_\alpha^{\hat c,D}u \,dx=\int_D u \eta\,dx.
\end{align*}
On the other hand
\begin{align*}
\int_D u \eta\,dx&=\int_D u R^D_\alpha(-A+\alpha )\eta\,dx
=\int_D(R^{D}_\alpha)^* u (-A+\alpha )\eta \,dx\\&=
\int_D (R^{D}_\alpha)^* u w\,dx=\int_D u  R^{D}_\alpha w\,dx.
\end{align*}
Since $u,w\in C_0(\overline D)$ were arbitrary, we get    \eqref{res-dual}.\qed}

\medskip

\subsection*{Acknowledgements}
{\small  T. Klimsiak is supported by Polish National Science Centre:
  \linebreak Grant No. 2017/25/B/ST1/00878.  T. Komorowski  acknowledges the
support of the  Polish  National Science Centre:
Grant No.  2016/23/B/ST1/00492. Both authors wish to express their
gratitude to   the referee for very careful reading of the manuscript and valuable remarks that 
lead  to significant improvement of the presentation. }

\end{document}